\documentclass[11pt]{elsarticle}

\usepackage{amssymb}
\usepackage{amsthm}
\usepackage{amsmath}

\usepackage{lineno}

\input{tcilatex}
\begin{document}

\begin{frontmatter}



\title{Stochastic variational inequalities with oblique subgradients\tnoteref{footnotetitle}}
\tnotetext[footnotetitle]{The work for this paper was supported by the Grant Marie Curie Initial Training Network (ITN) FP7-PEOPLE-2007-1-1-ITN, no.213841.}

\author[auth1]{Anouar M. Gassous}
\ead{masigassous@gmail.com}

\author[auth1,auth2]{Aurel R\u{a}\c{s}canu}
\ead{aurel.rascanu@uaic.ro}

\author[auth1]{Eduard Rotenstein\corref{correspondent}}
\ead{eduard.rotenstein@uaic.ro}

\address[auth1]{Faculty of Mathematics, ``Al. I. Cuza''\ University, Bd. Carol I, no. 9-11, Ia\c{s}i, Romania}
\address[auth2]{``Octav Mayer'' Mathematics Institute of the Romanian Academy, Bd. Carol I, no. 8, Ia\c{s}i, Romania}
\cortext[correspondent]{Corresponding author.}

\begin{abstract}
In this paper we will study the existence and uniqueness of the solution for
the stochastic variational inequality with oblique subgradients of the
following form:
$$\left\{\begin{array}{l}
dX_{t}+H\left(X_{t}\right)\partial\varphi\left(X_{t}\right)\left(dt\right)\ni f\left(t,X_{t}\right)dt+g\left(t,X_{t}\right)dB_{t},\quad
t>0,\smallskip \\
X_{0}=x\in\overline{\emph{Dom}\left(\varphi\right)}.
\end{array}
\right.$$
Here, the mixture between the monotonicity property of the subdifferential operator $\partial\varphi$ and the Lipschitz property of matrix mapping $X\longmapsto H(X)$ leads to stronger difficulties comparing to the classical case of stochastic variational inequalities. The existence result is based on a deterministic approach: a differential system with singular input is first analyzed.
\end{abstract}


\end{frontmatter}

\section{Introduction}
\label{int}

Since the early sixties, research has paid increasing attention to the study of reflected stochastic differential equations, the reflection process being approached in different ways. Skorohod, for instance, considered the problem of reflection for diffusion processes into a bounded domain (see, e.g., \cite{Skorohod62}). Tanaka focused on the problem of reflecting boundary conditions into convex sets for stochastic differential equations (see \cite{Tanaka78}). This kind of problem became the interest of many other authors, who considered that the state process is reflected by one or two reflecting barriers (see, e.g., \cite{Cepa:93}, \cite{Cepa:98}, \cite{Mckean63}, \cite{Hamadene/Hassani:2005} and the references therein). While, during the first studies, the trajectories of the system were reflected upon the normal direction, in 1984 Lions and Sznitman, in the paper \cite{Lions/Sznitman-84}, studied for the first time the following problem of oblique reflection in a domain:
\begin{equation}
\left\{
\begin{array}{l}
dX_{t}+dK_{t}=f\left(t,X_{t}\right) dt+g\left(t,X_{t}\right) dB_{t},\quad
t>0,\smallskip \\
X_{0}=x,\quad K_{t}={\displaystyle\int_{0}^{t}}1_{\left\{ X_{s}\in Bd(\emph{E})\right\}}\gamma(X_{s})d\left\updownarrow K\right\updownarrow_{s},
\end{array}
\right.\label{firstSkorohod problem}
\end{equation}
where, for the bounded oblique reflection $\gamma\in\mathcal{C}^{2}\left(\mathbb{R}^{d}\right)$, there exists a positive constant $\nu$ such that $\left(\gamma(x),n(x)\right)\geq\nu$, for every $x\in\emph{Bd}(E)$, $n(x)$ being the unit outward normal vector. A generalization, with
respect to the smoothness of the domain, of the result of Lions and Sznitman
was given after by Depuis and Ishi in the paper \cite{Dupuis/Ishii:93}. They
assumed that the domain in which we have the oblique reflection has some
additional regularity properties.

The aim of our paper consists in extending the problem of oblique reflection in the framework of deterministic and stochastic variational inequalities. This kind of multivalued stochastic differential equations were introduced in the literature by Asiminoaei \& R\u{a}\c{s}canu in \cite{Asiminoaei/Rascanu:97}, Barbu \& R\u{a}\c{s}canu in \cite{Barbu/Rascanu:97} and Bensoussan \& R\u{a}\c{s}canu in \cite{Bensoussan/Rascanu:97}. They proved the existence and uniqueness result for the case of stochastic variational differential systems involving
subdifferential operators and, even more, they provided approximation and splitting-up schemes for this type of equations. The general result, for stochastic differential equations governed by maximal monotone operators
\begin{equation}
\left\{
\begin{array}{l}
dX_{t}+A\left(X_{t}\right) \left(dt\right) \ni f\left(t,X_{t}\right)
dt+g\left(t,X_{t}\right) dB_{t},\smallskip \\
X_{0}=\xi,\ t\in \left[0,T\right]
\end{array}
\right.\nonumber
\end{equation}
was given by R\u{a}\c{s}canu in \cite{Rascanu:96}, the approach for proving
the existence and uniqueness being done via a deterministic multivalued
equation with singular input.

A different approach for solving these type of equations was introduced by R\u{a}\c{s}canu \& Rotenstein in the paper \cite{Rascanu/Rotenstein:11}. They reduced the existence problem for multivalued stochastic differential equations to a minimizing problem of a convex lower semicontinuous function. The solutions of these equations were identified with the minimum points of some suitably constructed convex lower semicontinuous functionals, defined on well chosen Banach spaces.

As the main objective of this paper we prove the existence and uniqueness of the solution for the following stochastic variational inequality
\begin{equation}
\left\{
\begin{array}{l}
dX_{t}+H\left(X_{t}\right) \partial \varphi \left(X_{t}\right)\left(dt\right)\ni f\left(t,X_{t}\right)dt+g\left(t,X_{t}\right)dB_{t},\quad
\quad t>0,\smallskip \\
X_{0}=x_{0},
\end{array}
\right.  \label{eqstgenerales}
\end{equation}
where $B$ is a standard Brownian motion defined on a complete probability
space and the new quantity that appears acts on the set of subgradients and
it will be called, from now on, \textit{oblique subgradient}. The problem becomes challenging due to the presence of this new term, which impose the use of some specific approaches because this new term preserve neither the monotony of the subdifferential operator nor the Lipschitz property of the matrix involved. First, we will
focus on the deterministic case, considering a generalized Skorohod problem
with oblique reflection of the form
\begin{equation}
\left\{
\begin{array}{l}
x\left(t\right) +{\displaystyle\int_{0}^{t}}H\left(x\left(s\right)
\right)dk\left(s\right)=x_{0}+{\displaystyle\int_{0}^{t}}f\left(s,x\left(s\right)\right)ds+m\left(t\right),\quad t\geq 0,\smallskip \\
dk\left(s\right)\in\partial\varphi\left(x\left(s\right)\right)
\left(ds\right),
\end{array}
\right.  \label{generalesko}
\end{equation}
where the singular input $m:\mathbb{R}_{+}\rightarrow\mathbb{R}^{d}$ is a continuous function. The existence results are obtained via Yosida penalization techniques.

The paper is organized as follows. Section 2 presents the notations and assumptions that will be used along this article and, also, a deterministic
generalized Skorohod problem with oblique reflection is constructed. The existence and uniqueness result for this problem can also be found here.
Section 3 is dedicated to the main result of our work; more precisely, the existence of a unique strong solution for our stochastic variational
inequa\-li\-ty with oblique subgradients is proved. The last part of the paper groups together some useful results that are used throughout this article.

\section{Generalized convex Skorohod problem with oblique subgradients}

\subsection{Notations. Hypotheses}

We first study the following deterministic generalized convex Skorohod problem with oblique subgradients:
\begin{equation}
\left\{
\begin{array}{l}
dx\left(t\right)+H\left(x\left(t\right)\right)\partial\varphi\left(x\left(t\right)\right)\left(dt\right)\ni dm\left(t\right),\quad
t>0,\smallskip \\
x\left( 0\right) =x_{0},
\end{array}
\right. \label{osp-eq1}
\end{equation}
where
\begin{equation}
\left\{
\begin{array}{rl}
\left(i\right)\quad & x_{0}\in Dom\left( \varphi \right) \overset{def}{=}\{x\in \mathbb{R}^{d}:\varphi(x)<\infty \},\medskip \\
\left(ii\right)\quad & m\in C\left(\mathbb{R}_{+};\mathbb{R}^{d}\right),\quad m\left( 0\right)=0,
\end{array}
\right.  \label{osp-h0}
\end{equation}
$H=\left(h_{i,j}\right)_{d\times d}\in C_{b}^{2}\left(\mathbb{R}^{d};
\mathbb{R}^{d\times d}\right)$ is a matrix, such that for all $x\in \mathbb{R}^{d}$,
\begin{equation}
\left\{
\begin{array}{rl}
\left(i\right)\quad & h_{i,j}\left(x\right)=h_{j,i}\left(x\right),\quad \textrm{for every\ }i,j\in \overline{1,d},\medskip \\
\left(ii\right)\quad & \dfrac{1}{c}\left\vert u\right\vert^{2}\leq
\left\langle H\left(x\right)u,u\right\rangle\leq c\left\vert u\right\vert
^{2},\quad \forall ~u\in \mathbb{R}^{d}\textrm{\ (for some\ }c\geq 1\textrm{)}
\end{array}
\right.  \label{osp-h0-A}
\end{equation}
and
\begin{equation}
\varphi:\mathbb{R}^{d}\rightarrow\left]-\infty,+\infty\right]\textrm{\ is
a proper l.s.c. convex function.}  \label{osp-h2}
\end{equation}
Denote by $\partial \varphi$ the subdifferential operator of $\varphi$:
$$\partial\varphi\left(x\right)\overset{def}{=}\left\{ \hat{x}\in \mathbb{R}^{d}:\left\langle \hat{x},y-x\right\rangle +\varphi \left( x\right) \leq
\varphi \left(y\right),\textrm{ for all }y\in \mathbb{R}^{d}\right\}$$ and
$Dom(\partial \varphi )=\{x\in \mathbb{R}^{d}:\partial \varphi (x)\neq\emptyset \}$. We will use the notation $(x,\hat{x})\in \partial \varphi$ in order to express that $x\in Dom(\partial \varphi )$ and $\hat{x}\in \partial \varphi (x)$.\medskip

The vector defined by the quantity $H\left(x\right) h$, with $h\in \partial
\varphi \left( x\right)$, will be called in what follows {\it oblique subgradient}.

\begin{remark}
If $E$ is a closed convex subset of $\mathbb{R}^{d}$, then
$$\varphi\left(x\right)=I_{E}\left(x\right)=\left\{\begin{array}{cc}
0, & \textrm{if }x\in E,\medskip \\
+\infty, & \textrm{if\ }x\notin E
\end{array}
\right.$$
is a convex l.s.c. function and, for $x\in E$,
$$\partial I_{E}\left(x\right)=\{\hat{x}\in \mathbb{R}^{d}:\left\langle \hat{x},y-x\right\rangle \leq 0,\;\forall ~y\in E\}=N_{E}\left(x\right),$$ where $N_{E}\left(x\right) $ is the closed external normal cone to $E$ at $x$. We have $N_{E}\left( x\right)=\emptyset$ if $x\notin E$ and $N_{E}\left(x\right)=\left\{0\right\}$ if $x\in int\left(E\right) $ (we denoted by $int\left(E\right)$ the interior of the set $E$).
\end{remark}

\begin{remark}
A vector $\nu_{x}$ associated to $x\in Bd\left(E\right)$ (we denoted by $Bd\left(E\right)$ the boundary of the set $E$) is called
external direction if there exists $\rho_{0}>0$ such that $x+\rho \nu
_{x}\notin E$ for all $0<\rho \leq \rho_{0}.$ In this case there exists $
c^{\prime }>0$, $n_{x}\in N_{E}\left(x\right),$ $\left\vert
n_{x}\right\vert=1$, such that $\left\langle n_{x},\nu_{x}\right\rangle
\geq c^{\prime}$. Remark that, if we consider the symmetric matrix
\begin{equation}
M\left( x\right) =\left\langle \nu _{x},n_{x}\right\rangle I_{d\times d}-\nu
_{x}\otimes n_{x}-n_{x}\otimes \nu _{x}+\frac{2}{\left\langle \nu
_{x},n_{x}\right\rangle }\nu _{x}\otimes \nu _{x}~,  \label{osp-A1}
\end{equation}
then%
$$\nu _{x}=M\left( x\right) n_{x}~,\textrm{ for all\ } x\in Bd\left(E\right).$$\smallskip
\end{remark}

Let $\left[ H\left( x\right) \right] ^{-1}$ be the inverse matrix of $H\left( x\right)$. Then $\left[ H\left(x\right)\right]^{-1}$ has the
same properties (\ref{osp-h0-A}) as $H\left( x\right)$. Denote
$$b=\sup\limits_{x,y\in \mathbb{R}^{d},~x\neq y}\frac{\left\vert H\left(
x\right) -H\left( y\right) \right\vert }{\left\vert x-y\right\vert }
+\sup\limits_{x,y\in \mathbb{R}^{d},~x\neq y}\frac{|\left[ H\left(x\right)
\right] ^{-1}-\left[ H\left(y\right)\right]^{-1}|}{\left\vert
x-y\right\vert},$$
where $\left\vert H\left( x\right) \right\vert \overset{def}{=}\left(\sum_{i,j=1}^{d}\left\vert h_{i,j}\left(x\right)\right\vert^{2}\right)
^{1/2}.$\medskip \newline
We shall call {\it oblique reflection} directions of the form
$$\nu_{x}=H\left( x\right) n_{x},\quad \textrm{with }x\in Bd\left(E\right),$$
where $n_{x}\in N_{E}\left(x\right)$.\medskip

If $E=\overline{E}\subset \mathbb{R}^{d}$ and $E^{c}=\mathbb{R}^{d}\setminus
E$, then we denote by
$$E_{\varepsilon }=\left\{x\in E:dist\left(x,E^{c}\right)\geq\varepsilon
\right\}=\overline{\left\{x\in E:B\left(x,\varepsilon \right)\subset
E\right\}}$$
the $\varepsilon-$interior of $E$.\medskip

We impose the following supplementary assumptions
\begin{equation}
\left\{
\begin{array}{rl}
\left(i\right) \quad & D=Dom\left(\varphi \right) \;\textrm{is a closed
subset of\ }\mathbb{R}^{d},\medskip \\
\left(ii\right) \quad & \exists \textrm{~}r_{0}>0,\;D_{r_{0}}\neq \emptyset
\quad \textrm{and}\quad h_{0}=\sup_{z\in D}dist\left( z,D_{r_{0}}\right)<\infty,\medskip \\
\left(iii\right) \quad & \exists ~L\geq 0,\;\textrm{such that }\left\vert
\varphi \left(x\right) -\varphi \left(y\right)\right\vert \leq
L+L\left\vert x-y\right\vert,\\ & \hfill \textrm{for all\ }x,y\in D.
\end{array}
\right.  \label{osp-h3}
\end{equation}

\noindent For example, condition (\ref{osp-h3}-($iii$)) is verified by functions $\varphi:\mathbb{R}^{d}\rightarrow\mathbb{R}$ of the following type:
\[
\varphi\left(x\right)=\varphi_{1}\left(x\right)+\varphi_{2}\left(x\right)+I_{D}\left(x\right)\text{,}
\]
where $D$ is a convex set satisfying (\ref{osp-h3}-($ii$)), $\varphi_{1}:\mathbb{R}^{d}\rightarrow\mathbb{R}$ is a convex lower semicontinuous function, $\varphi_{2}:D\rightarrow\mathbb{R}$ is a Lipschitz function and $I_{D}$ is the convex indicator of the set $D$.

\subsection{A generalized Skorohod problem}

In this section we present the notion of solution for the generalized convex Skorohod problem with oblique subgradients (\ref{osp-eq1}) and, also, we provide full proofs for its existence and uniqueness.\medskip

If $k:\left[ 0,T\right] \rightarrow \mathbb{R}^{d}$ and $\mathcal{D}\left[ 0,T\right] $ is the set of the partitions of the time interval $\left[ 0,T\right] $, of the form $\Delta =(0=t_{0}<t_{1}<...<t_{n}=T)$, we denote
\[
S_{\Delta }(k)=\sum\limits_{i=0}^{n-1}|k(t_{i+1})-k(t_{i})|
\]%
and $\left\updownarrow k\right\updownarrow _{T}\overset{def}{=}%
\sup\limits_{\Delta \in \mathcal{D}}S_{\Delta }(k)$. In the sequel we
consider the space of bounded variation functions $BV(\left[ 0,T\right] ;%
\mathbb{R}^{d})=\{k~|~k:\left[ 0,T\right] \rightarrow \mathbb{R}^{d},$ $%
\left\updownarrow k\right\updownarrow _{T}<\infty \}.$
Taking on the space of continuous functions $C\left(\left[ 0,T\right];\mathbb{R}^{d}\right)$ the usual norm
\[
\left\Vert y\right\Vert _{T}\overset{def}{=}\left\Vert y\right\Vert _{C(\left[ 0,T\right];\mathbb{R}^{d})}=\sup \left\{ \left\vert y\left( s\right) \right\vert :0\leq s\leq T\right\},
\]
then $(C(\left[ 0,T\right] ;\mathbb{R}^{d}))^{\ast }=\{k\in BV(\left[ 0,T\right]; \mathbb{R}^{d}):k(0)=0\}$. The duality between these spaces is given
by the Riemann--Stieltjes integral $\left( y,k\right) \mapsto
\int_{0}^{T}\left\langle y\left( t\right) ,dk\left( t\right) \right\rangle.$
We will say that a function $k\in BV_{loc}([0,+\infty \lbrack ;\mathbb{R}^{d})$ if, for every $T>0$, $k\in BV(\left[ 0,T\right] ;\mathbb{R}^{d})$.$\smallskip$

\begin{definition}
Given two functions $x,k:\mathbb{R}_{+}\rightarrow$ $\mathbb{R}^{d}$, we say
that $dk\left( t\right) \in\partial\varphi\left( x\left( t\right) \right)
\left( dt\right) $ if
$$\begin{array}{ll}
\left(a\right) & x,k:\mathbb{R}_{+}\rightarrow\mathbb{R}^{d}\textrm{\ are continuous,}\medskip\\
\left(b\right) & x\left(t\right)\in\overline{Dom\left(\varphi\right)},\medskip\\
\left(c\right) & k\in BV_{loc}\left([0,+\infty \lbrack;\mathbb{R}^{d}\right),k\left(0\right)=0,\medskip\\
\left(d\right) & \displaystyle\int_{s}^{t}\left\langle y\left(r\right)-x(r),dk\left(r\right)\right\rangle+\displaystyle\int
_{s}^{t}\varphi\left(x\left(r\right)\right) dr\leq\displaystyle\int_{s}^{t}\varphi\left(y\left(r\right)\right)dr,\smallskip \\
\multicolumn{1}{r}{} & \multicolumn{1}{r}{\textrm{\ for all\ }0\leq s\leq t\leq
T\textrm{\ and\ }y\in C\left(\left[0,T\right];\mathbb{R}^{d}\right).}
\end{array}$$
\end{definition}

\noindent We state that

\begin{definition}
\label{def1}A pair of functions $\left(x,k\right)$ is a solution of the Skorohod problem with $H-$oblique subgradients (\ref{osp-eq1}) (and we write $\left(x,k\right) \in \mathcal{SP}\left(H\partial \varphi ;x_{0},m\right)$) if $x,k:\mathbb{R}_{+}\rightarrow$ $\mathbb{R}^{d}$ are continuous functions
and
\begin{equation}
\left\{
\begin{array}{ll}
\left( i\right)\; & x\left(t\right) +\displaystyle\int_{0}^{t}H\left(x\left(r\right)\right)dk\left(r\right)=x_{0}+m\left(t\right),\quad
\forall ~t\geq 0,\smallskip \\
\left(ii\right)\; & dk\left(r\right) \in \partial\varphi\left(x\left(r\right)\right)\left(dr\right).
\end{array}
\right.  \label{osp-sp}
\end{equation}
\end{definition}

In Annex, Section 4.1., we present some lemmas with a priori estimates of the solutions $\left(x,k\right)\in \mathcal{SP}\left(H\partial \varphi;x_{0},m\right)$.
We here recall the result from Lemma \ref{oSP-l4-compact}.

\begin{proposition}
\label{p1-apri-estim}If $\left(x,k\right) \in \mathcal{SP}\left( H\partial
\varphi ;x_{0},m\right) $ then, under assumptions (\ref{osp-h0}), (\ref{osp-h0-A}), (\ref{osp-h2}) and (\ref{osp-h3}) there exists a constant $C_{T}\left(\left\Vert m\right\Vert_{T}\right)=C\left(T,\left\Vert m\right\Vert_{T},b,c,r_{0},h_{0}\right)$, increasing function with respect to $\left\Vert m\right\Vert_{T}$, such that, for all $0\leq s\leq t\leq T$,
\begin{equation}
\begin{array}{l}
(a)$\quad$\, \left\Vert x\right\Vert_{T}+\left\updownarrow
k\right\updownarrow_{T}\leq C_{T}\left(\left\Vert m\right\Vert_{T}\right),\medskip \\
(b)$\quad$\, \left\vert x\left(t\right)-x\left(s\right)
\right\vert +\left\updownarrow k\right\updownarrow_{t}-\left\updownarrow
k\right\updownarrow_{s}\leq C_{T}\left(\left\Vert m\right\Vert_{T}\right)
\times \sqrt{t-s+\mathbf{m}_{m}\left(t-s\right)},
\end{array}
\label{prop of the sol}
\end{equation}
where $\mathbf{m}_{m}$ represents the modulus of continuity of the continuous function $m$.
\end{proposition}

We renounce now at the restriction that the function $f$ is identically $0\ $%
and we consider the equation written under differential form%
\begin{equation}
\left\{
\begin{array}{l}
dx\left( t\right) +H\left( x\left( t\right) \right) \partial \varphi \left(
x\left( t\right) \right) \left( dt\right) \ni f\left( t,x\left( t\right)
\right) dt+dm\left( t\right),\quad t>0,\smallskip \\
x\left( 0\right)=x_{0},
\end{array}%
\right.  \label{ob5}
\end{equation}%
where%
\begin{equation}
\begin{array}{ll}
\left( i\right) \; & \left( t,x\right) \longmapsto f\left( t,x\right):%
\mathbb{R}_{+}\times \mathbb{R}^{d}\rightarrow \mathbb{R}^{d}\textrm{ is a
Carath\'{e}odory function} \\
& \quad \quad \quad \textrm{(i.e. measurable w.r. to }t\textrm{ and continuous
w.r. to }x\textrm{),} \\
\left( ii\right) \; & {\displaystyle\int_{0}^{T}}\left( f^{\#}\left(
t\right) \right) ^{2}dt<\infty ,\quad \textrm{where}\quad f^{\#}\left(
t\right) =\sup_{x\in Dom\left( \varphi \right) }\left\vert f\left(
t,x\right) \right\vert.%
\end{array}
\label{osp-h4}
\end{equation}

The estimates (\ref{prop of the sol}) hold too for a solution of Eq.(\ref{ob5}), but, now, the constant $C_{T}(\left\Vert m\right\Vert _{T})$ depends also on the quantity $\int_{0}^{T}f^{\#}(t)dt$. We are now able to formulate the main result of this section.

\begin{theorem}
\label{oSP-t1} Let the assumptions (\ref{osp-h0}), (\ref{osp-h0-A}), (\ref%
{osp-h2}), (\ref{osp-h3}) and (\ref{osp-h4}) be satisfied. Then the
differential equation (\ref{ob5}) has at least one solution in the sense of
Definition \ref{def1}, i.e. $x,k:\mathbb{R}_{+}\rightarrow $ $\mathbb{R}^{d}$
are continuous functions and%
\begin{equation}
\left\{
\begin{array}{rl}
\left(j\right) \; & x\left(t\right) +\displaystyle\int_{0}^{t}H\left(x\left(r\right) \right) dk\left(r\right) =x_{0}+\displaystyle
\int_{0}^{t}f\left(r,x\left(r\right)\right) dr+m\left(t\right),\quad
\forall ~t\geq 0,\smallskip \\
\left(jj\right)\; & dk\left(r\right)\in\partial\varphi\left(x\left(r\right)\right)\left(dr\right).
\end{array}
\right.  \label{osp-eq}
\end{equation}
\end{theorem}

\begin{proof}[{\bf Proof}]
We will divide the proof in two separate steps. First we will analyze the case of the regular function $m$ and, in the sequel, we consider the situation of the singular input $m$.\newline

\noindent{\bf{\it{Step 1.}} {\it Case}} $m\in C^{1}\left(\mathbb{R}_{+};
\mathbb{R}^{d}\right)$\medskip

It is sufficient to prove the existence of a solution on an interval $\left[0,T\right]$ arbitrary, fixed.

\noindent Let $n\in \mathbb{N}^{\ast }$, $n\geq T$, fixed, consider $\varepsilon =\frac{T}{n}$ and the extensions $f\left(s,x\right)=0$ and $m\left(s\right)=s \cdot m^{\prime}\left(0+\right)$ for $s<0$. Based on the notations from Annex 4.2., we consider the penalized problem
$$\begin{array}{l}
x_{\varepsilon}\left(t\right)=x_{0},\quad \textrm{if\ }t<0,\medskip \\
\multicolumn{1}{r}{x_{\varepsilon}\left(t\right)+\displaystyle
\int_{0}^{t}H\left( x_{\varepsilon}\left(s\right)\right) dk_{\varepsilon
}\left( s\right)=x_{0}+\displaystyle\int_{0}^{t}\left[ f\left(s-\varepsilon,\mathbb{\pi }_{D}\left(x_{\varepsilon}\left(s-\varepsilon
\right)\right)\right)+m^{\prime }\left(s-\varepsilon\right)\right]ds,}
\\ \multicolumn{1}{r}{\;t\in \left[0,T\right],}
\end{array}$$
or, equivalent,
\begin{equation}
\begin{array}{l}
x_{\varepsilon}\left(t\right)=x_{0},\quad \textrm{if\ }t<0,\medskip \\
x_{\varepsilon}\left(t\right)+\displaystyle\int_{0}^{t}H\left(x_{\varepsilon}\left(s\right)\right)\nabla \varphi
_{\varepsilon}\left(x_{\varepsilon}\left(s\right)\right)ds
=x_{0}+\displaystyle\int_{-\varepsilon}^{t-\varepsilon}\left[f\left(s,
\mathbb{\pi}_{D}\left(x_{\varepsilon}\left(s\right)\right)\right)
+m^{\prime}\left(s\right)\right]ds,\;\;t\in \left[0,T\right],
\end{array}
\label{ea-inte}
\end{equation}
where
\[
k_{\varepsilon }(t)=\int_{0}^{t}\nabla \varphi _{\varepsilon}(x_{\varepsilon }(s))ds
\]
and $\pi _{D}(x)$ is the projection of $x$ on the set $D=\overline{Dom(\varphi )}=Dom(\varphi )$, uniquely defined by $\pi _{D}(x)\in \mathcal{D}$ and $dist(x,D)=|x-\pi_{D}(x)|$.

Since $x\longmapsto H\left(x\right)\nabla\varphi_{\varepsilon}\left(x\right):\mathbb{R}^{d}\rightarrow\mathbb{R}^{d}$ is a sublinear and
locally Lipschitz continuous function and, for $s\leq t-\varepsilon$,
$$\left\vert f\left(s,\mathbb{\pi}_{D}\left(x_{\varepsilon}\left(s\right)
\right)\right) \right\vert \leq f^{\#}\left(s\right),$$
then, recursively, on the intervals $[i\varepsilon,\left(i+1\right)
\varepsilon]$ the approximating equation admits a unique solution $x_{\varepsilon}\in C\left(\left[0,T\right];\mathbb{R}^{d}\right)$. The regularity of the function $x\mapsto \left\vert x-a\right\vert
^{2}+\varphi _{\varepsilon }(x)$ and the definition of the approximating sequence $\{x_{\varepsilon }\}_{\varepsilon }$ implies that, for $u_{0}\in
Dom(\varphi )$, we have
\begin{equation}
\begin{array}{l}
\left\vert x_{\varepsilon}\left(t\right)-u_{0}\right\vert^{2}+\varphi_{\varepsilon}\left( x_{\varepsilon}\left(t\right)\right)\smallskip \\
\quad \quad \quad \quad +\displaystyle\int_{0}^{t}\left\langle H\left(x_{\varepsilon}\left(s\right)\right)\nabla \varphi_{\varepsilon}\left(x_{\varepsilon}\left(s\right)\right),2\left[x_{\varepsilon}\left(s\right)-u_{0}\right]+
\nabla\varphi_{\varepsilon}\left(x_{\varepsilon}\left(s\right)\right)\right\rangle ds\smallskip \\
\quad =\left\vert x_{0}-u_{0}\right\vert ^{2}+\varphi_{\varepsilon}\left(x_{0}\right)\smallskip \\
+\displaystyle\int_{0}^{t}\left\langle 2\left[ x_{\varepsilon}\left(s\right)-u_{0}\right] +\nabla\varphi_{\varepsilon }\left(x_{\varepsilon}\left(s\right)\right),f\left(s-\varepsilon,\mathbb{\pi}_{D}\left(x_{\varepsilon}
\left(s-\varepsilon\right)\right)\right)+m^{\prime}\left(s-\varepsilon\right)ds\right\rangle.
\end{array}
\label{ea-ini}
\end{equation}

Let consider an arbitrary fixed pair $\left(u_{0},\hat{u}_{0}\right)\in\partial \varphi$. Since $\nabla\varphi_{\varepsilon}\left(u_{0}\right)=\partial\varphi_{\varepsilon}
\left(u_{0}\right)$, then it is easy to verify, from the definition of the subdifferential operator, that
$$\left\vert \varphi_{\varepsilon}\left(x_{\varepsilon}\right)-\varphi
_{\varepsilon }\left(u_{0}\right)\right\vert +\varphi_{\varepsilon
}\left(u_{0}\right)-2\left\vert\nabla \varphi_{\varepsilon}\left(u_{0}\right)\right\vert \left\vert x_{\varepsilon }-u_{0}\right\vert \leq
\varphi_{\varepsilon}\left( x_{\varepsilon }\right).$$

\noindent Also, since $\nabla \varphi _{\varepsilon }(u_{0})\in \partial \varphi
(J_{\varepsilon }(u_{0}))$, where $J_{\varepsilon }(x)=x-\varepsilon \nabla
\varphi _{\varepsilon }(x)$, then%
\[
\left\langle \hat{u}_{0}-\nabla \varphi _{\varepsilon
}(u_{0}),u_{0}-(u_{0}-\varepsilon \nabla \varphi _{\varepsilon
}(u_{0})\right\rangle \geq 0,
\]%
which yields, after short computations, $|\nabla \varphi _{\varepsilon
}(u_{0})|\leq |\hat{u}_{0}|$. Moreover,%
\begin{eqnarray*}
-\varepsilon |\hat{u}_{0}|^{2} &\leq &-\varepsilon \left\langle \hat{u}%
_{0},\nabla \varphi _{\varepsilon }(u_{0})\right\rangle =\left\langle \hat{u}%
_{0},J_{\varepsilon }(u_{0})-u_{0}\right\rangle  \\
&\leq &\varphi (J_{\varepsilon }(u_{0}))-\varphi (u_{0}) \\
&\leq &\varphi _{\varepsilon }(u_{0})-\varphi (u_{0})\leq 0.
\end{eqnarray*}%
Due to%
\[
\varphi _{\varepsilon }(x_{\varepsilon }(t))\geq |\varphi _{\varepsilon
}(x_{\varepsilon }(t))-\varphi _{\varepsilon }(u_{0})|+\varphi (u_{0})-|\hat{%
u}_{0}|^{2}-2|\hat{u}_{0}||x_{\varepsilon }(t)-u_{0}|,
\]
from Eq.(\ref{ea-ini}) we obtain

\begin{equation}
\begin{array}{l}
\left\vert x_{\varepsilon }\left( t\right) -u_{0}\right\vert ^{2}+\left\vert
\varphi _{\varepsilon }\left( x_{\varepsilon }\left( t\right) \right)
-\varphi _{\varepsilon }\left( u_{0}\right) \right\vert \smallskip \\
\quad \quad \quad \quad +\displaystyle\int_{0}^{t}\left\langle H\left(
x_{\varepsilon }\left( s\right) \right) \nabla \varphi _{\varepsilon }\left(
x_{\varepsilon }\left( s\right) \right) ,2\left[ x_{\varepsilon }\left(
s\right) -u_{0}\right] +\nabla \varphi _{\varepsilon }\left( x_{\varepsilon
}\left( s\right) \right) \right\rangle ds\smallskip \\
\quad \leq \left\vert x_{0}-u_{0}\right\vert ^{2}+\varphi \left(
x_{0}\right) -\varphi \left( u_{0}\right) +\left\vert \hat{u}_{0}\right\vert
^{2}+2\left\vert \hat{u}_{0}\right\vert \left\vert x_{\varepsilon }\left(
t\right) -u_{0}\right\vert \smallskip \\
+\displaystyle\int_{0}^{t}\left\langle 2\left[ x_{\varepsilon }\left(
s\right) -u_{0}\right] +\nabla \varphi _{\varepsilon }\left( x_{\varepsilon
}\left( s\right) \right) ,f\left( s-\varepsilon,\mathbb{\pi }_{D}\left(
x_{\varepsilon }\left( s-\varepsilon \right) \right) \right)+m^{\prime
}\left( s-\varepsilon \right) \right\rangle ds.%
\end{array}
\label{ea-ini2}
\end{equation}%
Denoting by $C$ a generic constant independent of $\varepsilon$ ($C$
depends only of $c$ and $u_{0}$), the following estimates hold (to be
shortened we omit the argument $s$, writing $x_{\varepsilon}$ in the place
of $x_{\varepsilon }\left( s\right) $):\medskip

\begin{itemize}
\item $\dfrac{1}{c}\left\vert \nabla\varphi_{\varepsilon}\left(
x_{\varepsilon}\right) \right\vert ^{2}\leq\left\langle H\left(
x_{\varepsilon}\right) \nabla\varphi_{\varepsilon}\left( x_{\varepsilon
}\right) ,\nabla\varphi_{\varepsilon}\left( x_{\varepsilon}\right)
\right\rangle,$

\item\begin{align*}
\left\langle H\left( x_{\varepsilon }\right) \nabla \varphi _{\varepsilon
}\left( x_{\varepsilon }\right) ,2\left( x_{\varepsilon }-u_{0}\right)
\right\rangle & \geq -2\left\vert x_{\varepsilon }-u_{0}\right\vert
\left\vert H\left( x_{\varepsilon }\right) \nabla \varphi _{\varepsilon
}\left( x_{\varepsilon }\right) \right\vert \\
& \geq -2c\left\vert x_{\varepsilon }-u_{0}\right\vert \left\vert \nabla
\varphi _{\varepsilon }\left( x_{\varepsilon }\right) \right\vert \\
& \geq -C\sup_{r\leq s}\left\vert x_{\varepsilon }\left( r\right)
-u_{0}\right\vert ^{2}-\dfrac{1}{4c}\left\vert \nabla \varphi _{\varepsilon
}\left( x_{\varepsilon }\right) \right\vert ^{2},
\end{align*}

\item $2\left\vert \hat{u}_{0}\right\vert \left\vert x_{\varepsilon }\left(
t\right) -u_{0}\right\vert \leq \dfrac{1}{2}\sup\limits_{r\leq t}\left\vert
x_{\varepsilon }\left( r\right) -u_{0}\right\vert ^{2}+2\left\vert \hat{u}%
_{0}\right\vert ^{2},$

\item
\begin{align*}
& \left\langle 2\left( x_{\varepsilon }\left( s\right) -u_{0}\right) +\nabla
\varphi _{\varepsilon }\left( x_{\varepsilon }\left( s\right) \right),f\left(s-\varepsilon,\mathbb{\pi}_{D}\left(x_{\varepsilon}\left(
s-\varepsilon \right) \right) \right) +m^{\prime }\left( s-\varepsilon
\right) \right\rangle \\
& \leq \dfrac{1}{8c}\left\vert 2\left( x_{\varepsilon }\left( s\right)
-u_{0}\right) +\nabla \varphi _{\varepsilon }\left( x_{\varepsilon }\right)
\right\vert ^{2}+2c\left\vert f\left( s-\varepsilon ,\mathbb{\pi }_{D}\left(
x_{\varepsilon }\left( s-\varepsilon \right) \right) \right) +m^{\prime
}\left( s-\varepsilon \right) \right\vert ^{2} \\
& \leq \dfrac{1}{4c}\left\vert \nabla \varphi _{\varepsilon }\left(
x_{\varepsilon }\left( s\right) \right) \right\vert ^{2}+\dfrac{1}{c}%
\left\vert x_{\varepsilon }\left( s\right) -u_{0}\right\vert ^{2}+4c\left[
(f^{\#}\left( s-\varepsilon \right) )^{2}+\left\vert m^{\prime }\left(
s-\varepsilon \right) \right\vert ^{2}\right] .
\end{align*}
\end{itemize}

\noindent Using the above estimates in (\ref{ea-ini2}) we infer%
$$
\begin{array}{c}
\left\vert x_{\varepsilon }\left( t\right) -u_{0}\right\vert ^{2}+\left\vert
\varphi _{\varepsilon }\left( x_{\varepsilon }\left( t\right) \right)
-\varphi _{\varepsilon }\left( u_{0}\right) \right\vert +\dfrac{1}{2c}%
\displaystyle\int_{0}^{t}\left\vert \nabla \varphi _{\varepsilon }\left(
x_{\varepsilon }\left( r\right) \right) \right\vert ^{2}dr\smallskip \\
\leq \left\vert x_{0}-u_{0}\right\vert ^{2}+\varphi \left( x_{0}\right) -\varphi
\left( u_{0}\right) +3\left\vert \hat{u}_{0}\right\vert ^{2}\smallskip \\
+\dfrac{1}{2}\sup\limits_{\theta \leq t}\left\vert x_{\varepsilon }\left(
\theta \right) -u_{0}\right\vert ^{2}+4c\displaystyle\int_{0}^{t}\left[
\left( f^{\#}\left( r-\varepsilon \right) \right) ^{2}+\left\vert m^{\prime
}\left( r-\varepsilon \right) \right\vert ^{2}\right] dr\smallskip \\
+C\displaystyle\int_{0}^{t}\sup\limits_{\theta \leq r}\left\vert
x_{\varepsilon }\left( \theta \right) -u_{0}\right\vert ^{2}dr.%
\end{array}%
$$
We write the inequality for $s\in \left[ 0,t\right] $ and then we take the $%
\sup_{s\leq t}$. Hence
$$
\begin{array}{l}
\left\Vert x_{\varepsilon }-u_{0}\right\Vert _{t}^{2}+\sup\limits_{s\leq
t}\left\vert \varphi _{\varepsilon }\left( x_{\varepsilon }\left( s\right)
\right) -\varphi _{\varepsilon }\left( u_{0}\right) \right\vert +%
\displaystyle\int_{0}^{t}\left\vert \nabla \varphi _{\varepsilon }\left(
x_{\varepsilon }\left( r\right) \right) \right\vert ^{2}dr\smallskip \\
\leq 2\left[ \left\vert x_{0}-u_{0}\right\vert ^{2}+\varphi \left(
x_{0}\right) -\varphi \left( u_{0}\right) +\left\vert \hat{u}_{0}\right\vert
^{2}\right] \smallskip \\
\quad +8c\displaystyle\int_{-1}^{t}\left[ \left( f^{\#}\left( r\right)
\right) ^{2}+\left\vert m^{\prime }\left( r\right) \right\vert ^{2}\right]
dr+C\displaystyle\int_{0}^{t}\left\Vert x_{\varepsilon }-u_{0}\right\Vert
_{r}^{2}dr.%
\end{array}%
$$
By the Gronwall inequality we have%
$$
\left\Vert x_{\varepsilon }-u_{0}\right\Vert _{t}^{2}\leq Ce^{Ct}\left[ %
\left[ \left\vert x_{0}-u_{0}\right\vert ^{2}+\varphi \left( x_{0}\right)
-\varphi \left( u_{0}\right) +\left\vert \hat{u}_{0}\right\vert ^{2}\right] +%
\displaystyle\int_{-1}^{t}\left[ (f^{\#}\left( r\right) )^{2}+\left\vert
m^{\prime }\left( r\right) \right\vert ^{2}\right] dr\right].
$$
Hence, there exists a constant $C_{T},$ independent of $\varepsilon,$ such
that%
\begin{equation}
\sup_{t\in \left[ 0,T\right] }\left\vert x_{\varepsilon }\left( t\right)
\right\vert ^{2}+\sup_{t\in \left[ 0,T\right] }\left\vert \varphi
_{\varepsilon }\left( x_{\varepsilon }\left( t\right) \right) \right\vert +%
\displaystyle\int_{0}^{T}\left\vert \nabla \varphi _{\varepsilon }\left(
x_{\varepsilon }\left( s\right) \right) \right\vert ^{2}ds\leq C_{T}~.
\label{est1}
\end{equation}%
Since $\nabla \varphi _{\varepsilon }\left( x\right) =\dfrac{1}{\varepsilon }%
\left( x-J_{\varepsilon }x\right),$ then, we also obtain%
\begin{equation}
\displaystyle\int_{0}^{T}\left\vert x_{\varepsilon }\left( s\right)
-J_{\varepsilon }\left( x_{\varepsilon }\left( s\right) \right) \right\vert
^{2}ds\leq \varepsilon C_{T}.  \label{est11}
\end{equation}%
From the approximating equation, for all $0\leq s\leq t\leq T,$ we have%
\begin{align*}
& \left\vert x_{\varepsilon }\left( t\right) -x_{\varepsilon }\left(
s\right) \right\vert \\
& \leq \left\vert \displaystyle\int_{s}^{t}H\left( x_{\varepsilon }\left(
r\right) \right) \nabla \varphi _{\varepsilon }\left( x_{\varepsilon }\left(
r\right) \right) dr\right\vert +\left\vert \displaystyle\int_{s-\varepsilon
}^{t-\varepsilon }f\left( r,\mathbb{\pi }_{D}\left( x_{\varepsilon }\left(
r\right) \right) \right) dr\right\vert \\
& +\left\vert m\left( t-\varepsilon \right) -m\left( s-\varepsilon \right)
\right\vert \\
& \leq c\displaystyle\int_{s}^{t}\left\vert \nabla \varphi _{\varepsilon
}\left( x_{\varepsilon }\left( r\right) \right) \right\vert dr+\displaystyle%
\int_{s-\varepsilon }^{t-\varepsilon }f^{\#}\left( r\right) dr+\mathbf{m}%
_{m}\left( t-s\right) \\
& \leq c\sqrt{t-s}\left( \displaystyle\int_{s}^{t}\left\vert \nabla \varphi
_{\varepsilon }\left( x_{\varepsilon }\left( r\right) \right) \right\vert
^{2}dr\right) ^{1/2}+\sqrt{t-s}\left( \displaystyle\int_{s-\varepsilon
}^{t-\varepsilon }(f^{\#}\left( r\right) )^{2}dr\right) ^{1/2}+\mathbf{m}%
_{m}\left( t-s\right) \\
& \leq C_{T}^{\prime }\left[ \sqrt{t-s}+\mathbf{m}_{m}\left( t-s\right) %
\right].
\end{align*}%
In fact, moreover we have%
\begin{align*}
\left\updownarrow x_{\varepsilon }\right\updownarrow _{\left[ s,t\right] }&
\leq \displaystyle\int_{s}^{t}\left\vert H\left( x_{\varepsilon }\left(
r\right) \right) \nabla \varphi _{\varepsilon }\left( x_{\varepsilon }\left(
r\right) \right) \right\vert dr+\displaystyle\int_{s-\varepsilon
}^{t-\varepsilon }\left\vert f\left( r,\mathbb{\pi }_{D}\left(
x_{\varepsilon }\left( r\right) \right) \right) \right\vert dr+\displaystyle%
\int_{s-\varepsilon }^{t-\varepsilon }\left\vert m^{\prime }\left( r\right)
\right\vert dr \\
& \leq C_{T}\sqrt{t-s}.
\end{align*}%
Hence $\left\{ x_{\varepsilon}:\varepsilon \in (0,1]\right\}$ is a bounded
and uniformly equicontinuous subset of $C\left( \left[ 0,T\right];\mathbb{R}%
^{d}\right)$. From Ascoli-Arzel\`{a}'s theorem it follows that there exists
$\varepsilon_{n}\rightarrow 0$ and $x\in C\left( \left[ 0,T\right];\mathbb{%
R}^{d}\right)$ such that%
$$
\lim_{n\rightarrow \infty }\left[ \sup_{t\in \left[ 0,T\right] }\left\vert
x_{\varepsilon _{n}}\left( t\right) -x\left( t\right) \right\vert \right]=0.
$$

By (\ref{est11}), there exists $h\in L^{2}\left( 0,T;\mathbb{R}^{d}\right) $
such that, on a subsequence, denoted also $\varepsilon_{n},$ we have
$$
J_{\varepsilon _{n}}\left( x_{\varepsilon _{n}}\right) \rightarrow x\quad
\textrm{in }L^{2}(0,T;\mathbb{R}^{d})\textrm{ and }a.e.\textrm{ in }\left[ 0,T%
\right],\quad \textrm{as }\varepsilon _{n}\rightarrow 0
$$ and
$$
\nabla \varphi_{\varepsilon _{n}} \left( x_{\varepsilon _{n}}\right) \rightharpoonup h,\quad
\textrm{weakly in }L^{2}(0,T;\mathbb{R}^{d}).
$$
Therefore, for all $t\in \left[ 0,T\right] $,%
\begin{equation}
\lim_{n\rightarrow \infty }\int_{0}^{t}H(x_{\varepsilon _{n}}(s))\nabla
\varphi _{\varepsilon _{n}}(x_{\varepsilon
_{n}}(s))ds=\int_{0}^{t}H(x(s))h(s)ds. \label{limhh}
\end{equation}
The lower semicontinuity property of $\varphi$ yields, a.e. $t\in \left[ 0,T%
\right],$%
$$
\varphi \left( x\left( t\right) \right) \leq \liminf_{n\rightarrow +\infty
}\varphi \left( J_{\varepsilon _{n}}\left( x_{\varepsilon _{n}}\left(
t\right) \right) \right) \leq \liminf_{n\rightarrow +\infty }\varphi
_{\varepsilon _{n}}\left( x_{\varepsilon _{n}}\left( t\right) \right) \leq
C_{T}~.
$$
Since $\nabla \varphi_{\varepsilon} \left( x_{\varepsilon }\right) \in \partial \varphi
\left( J_{\varepsilon }\left( x_{\varepsilon }\right) \right),$ then for
all $y\in C\left( \left[ 0,T\right];\mathbb{R}^{d}\right),$
$$
\displaystyle\int_{s}^{t}\left\langle \nabla \varphi_{\varepsilon} \left( x_{\varepsilon
}\left( r\right) \right) ,y\left( r\right) -J_{\varepsilon }(x_{\varepsilon }(r))
\right\rangle dr+\displaystyle\int_{s}^{t}\varphi \left( J_{\varepsilon
}\left( x_{\varepsilon }\left( r\right) \right) \right) dr\leq \displaystyle%
\int_{s}^{t}\varphi \left( y\left( r\right) \right) dr;
$$
passing to $\liminf_{\varepsilon _{n}\rightarrow 0}$ we obtain%
$$
\displaystyle\int_{s}^{t}\left\langle h\left( r\right),y\left( r\right)
-x\left( r\right) \right\rangle dr+\displaystyle\int_{s}^{t}\varphi \left(
x\left( r\right) \right) dr\leq \displaystyle\int_{s}^{t}\varphi \left(
y\left( r\right) \right) dr,
$$
for all $0\leq s\leq t\leq T$ and $y\in C\left( \left[ 0,T\right] ;\mathbb{R}%
^{d}\right)$, that is $h\left( r\right) \in \partial \varphi \left( x\left(
r\right) \right) $ a.e. $t\in \left[0,T\right].\smallskip$

Finally, taking into account (\ref{limhh}), by passing to limit for $\varepsilon =\varepsilon_{n}\rightarrow 0$ in the approximating equation (\ref{ea-inte}), via the Lebesgue dominated convergence theorem for the integral from the right-hand side, we get
$$x\left(t\right)+\displaystyle\int_{0}^{t}H\left(x\left(s\right)\right)
dk\left(s\right)=x_{0}+\displaystyle\int_{0}^{t}f\left(s,x\left(s\right)
\right)ds+m\left(t\right),$$
where
$$k\left( t\right)=\displaystyle\int_{0}^{t}h\left(s\right)ds.$$

\noindent {\it{Step 2.}} {\it{Case}} $m\in C\left(\left[0,T\right];\mathbb{R}^{d}\right).$\medskip

Let extend again $m\left(s\right)=0$ for $s\leq 0$ and define
$$m_{\varepsilon}\left(t\right)=\dfrac{1}{\varepsilon}\displaystyle
\int_{t-\varepsilon}^{t}m\left(s\right) ds=\dfrac{1}{\varepsilon}
\displaystyle\int_{0}^{\varepsilon }m\left( t+r-\varepsilon\right)dr,$$
We have
$$m_{\varepsilon}\in C^{1}(\left[0,T\right];\mathbb{R}^{d}),\quad
\left\Vert m_{\varepsilon}\right\Vert _{T}\leq \left\Vert m\right\Vert
_{T}\quad \textrm{and}\quad \mathbf{m}_{m_{\varepsilon}}\left(\delta \right)
\leq \mathbf{m}_{m}\left(\delta \right).
$$
Let $(x_{\varepsilon },k_{\varepsilon })$ be a solution of the approximating equation
$$
\left\{
\begin{array}{l}
x_{\varepsilon}\left(t\right)+\displaystyle\int_{0}^{t}H\left(x_{\varepsilon}\left(r\right)\right)dk_{\varepsilon}\left(r\right)
=x_{0}+\displaystyle\int_{0}^{t}f\left(r,x_{\varepsilon}\left(r\right)
\right)dr+m_{\varepsilon}\left(t\right),\;t\geq 0,\smallskip \\
dk_{\varepsilon}\left(r\right)\in\partial \varphi \left(x_{\varepsilon}\left(r\right)\right)\left(dr\right),
\end{array}
\right.
$$
solution which exists according to the first step of the proof. We have
$$
k_{\varepsilon}\left(t\right)=\displaystyle\int_{0}^{t}h_{\varepsilon}\left(s\right)ds,\quad h_{\varepsilon}\in L^{2}(0,T;\mathbb{R}^{d}),
$$
and
\begin{equation}
\displaystyle\int_{s}^{t}\left\langle y\left(r\right)-x_{\varepsilon}\left(r\right),dk_{\varepsilon}\left(r\right)\right\rangle+
\displaystyle\int_{s}^{t}\varphi\left(x_{\varepsilon}\left(r\right)
\right)dr\leq\displaystyle\int_{s}^{t}\varphi\left(y\left(r\right)
\right)dr,  \label{ob-7}
\end{equation}
for all $0\leq s\leq t\leq T$ and $y\in C\left(\left[0,T\right];\mathbb{R}^{d}\right).\smallskip$

From Lemma \ref{oSP-l4-compact}, with $m$ replaced by
$$M_{\varepsilon }\left(t\right)=\displaystyle\int_{0}^{t}f\left(r,x_{\varepsilon}\left(r\right)\right) dr+m_{\varepsilon }\left( t\right),$$
we have%
\begin{align*}
\left\Vert x_{\varepsilon}\right\Vert _{T}+\left\updownarrow k_{\varepsilon
}\right\updownarrow _{T}& \leq C_{T}\left( \left\Vert M_{\varepsilon
}\right\Vert _{T}\right) \quad \textrm{and} \\
\left\vert x_{\varepsilon }\left( t\right)-x_{\varepsilon }\left( s\right)
\right\vert +\left\updownarrow k_{\varepsilon }\right\updownarrow
_{t}-\left\updownarrow k_{\varepsilon }\right\updownarrow _{s}& \leq
C_{T}\left( \left\Vert M_{\varepsilon }\right\Vert_{T}\right) \times \sqrt{%
\mathbf{\mu}_{M_{\varepsilon }}\left(t-s\right)},
\end{align*}%
where, for $\delta >0$, $\mu _{g}(\delta )\overset{def}{=}\delta +\mathbf{m}_{g}(\delta )$ and $\mathbf{m}_{g}$ is the modulus of continuity of the continuous function
$g:\left[ 0,T\right] \rightarrow \mathbb{R}^{d}$ (for more details see Annex 4.1.). Since, for all $0\leq s\leq t\leq T,$%
\begin{align*}
\mathbf{\mu }_{M_{\varepsilon }}\left( t-s\right) & \leq t-s+\sqrt{t-s}%
\displaystyle\int_{0}^{T}(f^{\#}\left( r\right) )^{2}dr+\mathbf{m}_{m}\left(
t-s\right) \overset{def}{=}\gamma \left( t-s\right) \quad \textrm{and} \\
\left\Vert M_{\varepsilon }\right\Vert _{T}& =\mathbf{m}_{M_{\varepsilon
}}\left( T\right) \leq \displaystyle\int_{0}^{T}f^{\#}\left( r\right)
dr+\left\Vert m\right\Vert _{T}\overset{def}{=}\gamma_{T},
\end{align*}%
then there exist the positive constants $C_{T}(\gamma_{T})$ and $\tilde{C}_{T}(\gamma_{T})$ such that%
\begin{align*}
\left\Vert x_{\varepsilon }\right\Vert _{T}+\left\updownarrow k_{\varepsilon
}\right\updownarrow _{T}& \leq C_{T}\left(\gamma_{T}\right) \quad \textrm{and} \\
\mathbf{m}_{x_{\varepsilon }}\left( t-s\right) +\left\updownarrow
k_{\varepsilon }\right\updownarrow _{t}-\left\updownarrow k_{\varepsilon
}\right\updownarrow _{s}& \leq \tilde{C}_{T}\left(\gamma_{T}\right) \times \sqrt{%
\gamma \left( t-s\right)}.
\end{align*}%
By Ascoli-Arzel\`{a}'s theorem it follows that there exists $\varepsilon
_{n}\rightarrow 0$ and $x,k\in C\left( \left[ 0,T\right];\mathbb{R}%
^{d}\right) $ such that%
$$
x_{\varepsilon _{n}}\rightarrow x\quad \textrm{and}\quad \textrm{ }%
k_{\varepsilon _{n}}\rightarrow k\quad \textrm{in }C(\left[ 0,T\right];%
\mathbb{R}^{d}).
$$
Moreover, since $\left\updownarrow \cdot \right\updownarrow:C\left( \left[
0,T\right];\mathbb{R}^{d}\right) \rightarrow \mathbb{R}$ is a lower
semicontinuous function, then%
$$
\left\updownarrow k\right\updownarrow _{T}\leq \liminf_{n\rightarrow +\infty
}\left\updownarrow k_{\varepsilon _{n}}\right\updownarrow _{T}\leq C_{T,m}~.
$$
By Helly-Bray theorem, we can pass to the limit and we have, for all $0\leq s\leq t\leq T$,%
\[
\lim_{n\rightarrow \infty }\int_{s}^{t}\left\langle y(r)-x_{\varepsilon_{n}}(r),dk_{\varepsilon _{n}}(r)\right\rangle =\int_{s}^{t}\left\langle y(r)-x(r),dk(r)\right\rangle
\]
Passing now to $\liminf_{n\rightarrow +\infty }$ in (\ref{ob-7}) we infer $%
dk\left( r\right) \in \partial \varphi \left( x\left( r\right) \right)
\left( dr\right)$. Finally, taking $\lim_{n\rightarrow \infty }$ in the
approximating equation we obtain that $\left( x,k\right)$ is a solution of
the equation (\ref{osp-eq}). The proof is now complete.\hfill
\end{proof}

In the next step we will show in which additional conditions the equation (\ref{ob5}) admits a unique solution.\medskip

\begin{proposition}
\label{oSP-p1-uniq}Let the assumptions (\ref{osp-h0-A}), (\ref{osp-h0}), (%
\ref{osp-h2}), (\ref{osp-h3}) and (\ref{osp-h4}) be satisfied. Assume also that
there exists $\mu \in L_{loc}^{1}\left( \mathbb{R}_{+};\mathbb{R}_{+}\right)$ such, that for all $x,y\in \mathbb{R}^{d},$%
\begin{equation}
\left\vert f\left( t,x\right) -f\left( t,y\right) \right\vert \leq \mu
\left( t\right) \left\vert x-y\right\vert,\quad a.e.\ t\geq 0.
\label{osp-h5}
\end{equation}%
If $m\in BV_{loc}\left( \mathbb{R}_{+};\mathbb{R}^{d}\right) $, then the
generalized convex Skorohod problem with oblique subgradients (\ref{osp-eq1}%
) admits a unique solution $(x,k)$ in the space $C(%
\mathbb{R}_{+};\mathbb{R}^{d})\times \lbrack C(\mathbb{R}_{+};\mathbb{R}^{d})\cap BV_{loc}(\mathbb{R}_{+};\mathbb{R}^{d})]$.
Moreover, if $\left( x,k\right)$ and $(\hat{x},\hat{k})$ are two solutions, corresponding to $m$, respectively $\hat{m},$
then%
\begin{equation}
\left\vert x\left( t\right) -\hat{x}\left( t\right) \right\vert \leq
Ce^{CV\left( t\right) }\left[ \left\vert x_{0}-\hat{x}_{0}\right\vert
+\left\updownarrow m-\hat{m}\right\updownarrow _{t}\right],  \label{ob8}
\end{equation}%
where $V\left( t\right) =\left\updownarrow x\right\updownarrow
_{t}+\updownarrow \!\hat{x}\!\updownarrow _{t}+\left\updownarrow
k\right\updownarrow _{t}+\updownarrow\hat{k}\updownarrow_{t}+\displaystyle%
\int_{0}^{t}\mu \left( r\right) dr$ and $C$ is a constant depending only on $b$ and $c$.
\end{proposition}

\begin{proof}[{\bf Proof}]
The existence was proved in Theorem \ref{oSP-t1}. Let us prove the
inequality (\ref{ob8}) which clearly yields the uniqueness.

Consider the symmetric and strict positive matrix $Q\left( r\right) =\left[
H\left( x\left( r\right) \right) \right] ^{-1}+\left[ H\left( \hat{x}\left(
r\right) \right) \right] ^{-1}$. Remark that%
\begin{equation}
\begin{array}{l}
Q\left( r\right) \left[ H\left( \hat{x}\left( r\right) \right) d\hat{k}%
\left( r\right) -H\left( x\left( r\right) \right) dk\left( r\right) \right]
\\
=\left( \left[ H\left( x\left( r\right) \right) \right] ^{-1}-\left[ H\left(
\hat{x}\left( r\right) \right) \right] ^{-1}\right) \left[ H\left( \hat{x}%
\left( r\right) \right) d\hat{k}\left( r\right) +H\left( x\left( r\right)
\right) dk\left( r\right) \right] \\
+2\left[ d\hat{k}\left( r\right) -dk\left( r\right) \right].%
\end{array}
\label{eui_1}
\end{equation}%
Let $u\left( r\right) =Q^{1/2}\left( r\right) \left( x\left( r\right) -\hat{x%
}\left( r\right) \right).$ Then%
\begin{align*}
du\left( r\right) & =\left[ dQ^{1/2}\left( r\right) \right] \left( x\left(
r\right) -\hat{x}\left( r\right) \right) +Q^{1/2}\left( r\right) d\left[
x\left( r\right) -\hat{x}\left( r\right) \right] \\
& =\left[ \alpha \left( r\right) dx\left( r\right) +\hat{\alpha}\left(
r\right) d\hat{x}\left( r\right) \right] \left( x\left( r\right) -\hat{x}%
\left( r\right) \right) \\
& +Q^{1/2}\left( r\right) \left[ f\left( r,x\left( r\right) \right) -f\left(
r,\hat{x}\left( r\right) \right) \right] dr \\
& +Q^{1/2}\left( r\right) \left[ dm\left( t\right) -d\hat{m}\left( t\right) %
\right] \\
& +Q^{1/2}\left( r\right) \left[ -H\left( x\left( r\right) \right) dk\left(
r\right) +H\left( \hat{x}\left( r\right) \right) d\hat{k}\left( r\right) %
\right],
\end{align*}%
with $\alpha$, $\hat{\alpha}\in \mathcal{L(%
\mathbb{R}_{+}};\mathbb{R}^{d\times d}\mathcal{)}$, where $\mathcal{L(%
\mathbb{R}_{+}};\mathbb{R}^{d\times d}\mathcal{)}$ is the space of
continuous linear operators from $\mathbb{R}_{+}$ into $\mathbb{R}%
^{d\times d}$.

Using (\ref{eui_1}) and the assumptions on the matrix-valued functions $%
x\longmapsto H\left( x\right) $ and $x\longmapsto \left[ H\left( x\right) %
\right]^{-1},$ we have (as signed measures on $\mathbb{R}_{+}$), for some
positive constants $C_{1},C_{2},C_{3},C$ depending only on the constants $c$
and $b,$%
\begin{align*}
\left\langle u\left( r\right) ,du\left( r\right) \right\rangle & \leq
C_{1}\left\vert u\left( r\right) \right\vert ^{2}\left( d\left\updownarrow
x\right\updownarrow _{r}+d\updownarrow \!\hat{x}\!\updownarrow _{r}\right)
+C_{2}\mu \left( r\right) \left\vert u\left( r\right) \right\vert ^{2}dr \\
& +C_{3}\left\vert u\left( r\right) \right\vert d\left\updownarrow m-\hat{m}%
\right\updownarrow _{r} \\
& +\left\langle x\left( r\right) -\hat{x}\left( r\right) ,Q\left( r\right) %
\left[ H\left( \hat{x}\left( r\right) \right) d\hat{k}\left( r\right)
-H\left( x\left( r\right) \right) dk\left( r\right) \right] \right\rangle  \\
& \leq C\left\vert u\left( r\right) \right\vert d\left\updownarrow m-\hat{m}%
\right\updownarrow _{r}+C\left\vert u\left( r\right) \right\vert
^{2}dV\left( r\right),
\end{align*}%
with $V\left( t\right) =\left\updownarrow x\right\updownarrow
_{t}+\updownarrow \!\hat{x}\!\updownarrow _{t}+\left\updownarrow
k\right\updownarrow _{t}+\updownarrow \hat{k}\updownarrow _{t}+\displaystyle%
\int_{0}^{t}\mu \left( r\right) dr.$ Now, by (\ref{AnC-dxRNV}), we infer, for
all $t\geq 0$,%
\begin{equation}
\left\vert u\left( t\right) \right\vert \leq e^{CV\left( t\right)
}\left\vert x_{0}-\hat{x}_{0}\right\vert +\displaystyle\int_{0}^{t}Ce^{C%
\left[ V\left( t\right) -V\left( r\right) \right] }d\left\updownarrow m-\hat{%
m}\right\updownarrow _{r}~.  \label{o2unq}
\end{equation}%
and the inequality (\ref{ob8}) follows.\hfill
\end{proof}

\begin{proposition}
\label{p-aproxm}Under the assumptions of Proposition \ref{oSP-p1-uniq} and,
for $m\in C^{1}\left( \mathbb{R}_{+};\mathbb{R}^{d}\right)$, the solution $%
\left( x_{\varepsilon }\right) _{0<\varepsilon \leq 1}$ of the approximating
equation%
\begin{equation}
\begin{array}{l}
x_{\varepsilon }\left( t\right) +\displaystyle\int_{0}^{t}H\left(
x_{\varepsilon }\left( s\right) \right) dk_{\varepsilon }\left( s\right)
=x_{0}+\displaystyle\int_{0}^{t}f\left( s,\mathbb{\pi }_{D}\left(
x_{\varepsilon }\left( s\right) \right) \right) ds+m\left( t\right),\quad
t\geq 0,\smallskip \\
dk_{\varepsilon }\left( s\right) =\nabla \varphi _{\varepsilon }\left(
x_{\varepsilon }\left( s\right) \right) ds,%
\end{array}
\label{aproxeq}
\end{equation}%
has the following properties:\newline
$\bullet $ for all $T>0$ there exists a constant $C_{T}$, independent of $%
\varepsilon,\delta \in ]0,1],$ such that%
$$
\begin{array}{rl}
\left( j\right) \quad & \sup\limits_{t\in \left[ 0,T\right] }\left\vert
x_{\varepsilon }\left( t\right) \right\vert ^{2}+\sup_{t\in \left[ 0,T\right]
}\left\vert \varphi _{\varepsilon }\left( x_{\varepsilon }\left( t\right)
\right) \right\vert +\displaystyle\int_{0}^{T}\left\vert \nabla \varphi
_{\varepsilon }\left( x_{\varepsilon }\left( s\right) \right) \right\vert
^{2}ds\leq C_{T}~,\medskip \\
\left( jj\right) \quad & \left\updownarrow x_{\varepsilon
}\right\updownarrow _{\left[ s,t\right] }\leq C_{T}\sqrt{t-s},\quad \textrm{%
for all }0\leq s\leq t\leq T~,\medskip \\
\left( jjj\right) \quad & \left\Vert x_{\varepsilon }-x_{\delta }\right\Vert
_{T}\leq C_{T}\sqrt{\varepsilon +\delta }~.%
\end{array}%
$$
$\bullet$ Moreover, there exist $x,k\in C\left( \left[ 0,T\right];\mathbb{R%
}^{d}\right)$ and $h\in L^{2}\left( 0,T;\mathbb{R}^{d}\right) ,$ such that%
\newline
$$
\lim_{\varepsilon \rightarrow 0}k_{\varepsilon }\left( t\right) =k\left(
t\right) =\int_{0}^{t}h\left( s\right) ds\textrm{, for all }t\in \left[ 0,T%
\right],
$$
$$\lim\limits_{\varepsilon \rightarrow 0}\left\Vert x_{\varepsilon}-x\right\Vert_{T}=0$$
and $\left(x,k\right)$ is the unique solution of the variational
inequality with oblique subgradients (\ref{osp-eq}).
\end{proposition}

\begin{proof}[{\bf Proof}]
The proof for the estimates $\left(j\right)$ and $\left(jj\right)$ are
exactly as in the proof of Theorem \ref{oSP-t1}.

Let us prove $\left(jjj\right)$. Similarly to the proof of the uniqueness
result (Proposition \ref{oSP-p1-uniq}), we introduce
$Q_{\varepsilon,\delta}\left(s\right)=\left[H\left(x_{\varepsilon}\left(s\right)\right)\right]^{-1}+\left[H\left(x_{\delta}\left(s\right)\right)\right]^{-1}$. Once again, to simplify the reading, we
omit $s$ in the argument of $x_{\varepsilon}\left(s\right)$ and $x_{\delta}\left(s\right)$. Remark that
\begin{align*}
Q_{\varepsilon,\delta}\left(s\right) & \left[H\left(x_{\delta}\right)
\nabla\varphi_{\delta}\left(x_{\delta}\right)-H\left(x_{\varepsilon}\right)\nabla\varphi_{\varepsilon}\left(x_{\varepsilon}\right)\right]
\\
&=\left(\left[H\left(x_{\varepsilon}\right)\right]^{-1}-\left[H\left(x_{\delta}\right)\right]^{-1}\right)\left[H\left(x_{\delta}\right)\nabla\varphi _{\delta}\left(x_{\delta}\right)
+H\left(x_{\varepsilon}\right)\nabla\varphi_{\varepsilon}\left(x_{\varepsilon}\right)\right]\\
& +2\left[dk_{\delta}\left(s\right)-dk_{\varepsilon}\left(s\right)\right].
\end{align*}

Let $u_{\varepsilon,\delta}\left(s\right)=Q_{\varepsilon,\delta}^{1/2}\left(s\right)\left(x_{\varepsilon}\left(s\right)-x_{\delta}\left(s\right)\right)$. Then
\begin{align*}
du_{\varepsilon,\delta}\left(s\right)&=\left[ dQ_{\varepsilon,\delta
}^{1/2}\left( s\right) \right] \left(x_{\varepsilon}-x_{\delta}\right)
+Q_{\varepsilon,\delta}^{1/2}\left(s\right) d\left[x_{\varepsilon
}-x_{\delta}\right] \\
& =\left[\alpha_{\varepsilon,\delta}\left( s\right) dx_{\varepsilon
}+\beta_{\varepsilon,\delta}\left( s\right) dx_{\delta}\right] \left(
x_{\varepsilon}-x_{\delta}\right) \\
& +Q_{\varepsilon,\delta}^{1/2}\left( s\right) \left[ f\left( s,\mathbb{\pi}%
_{D}\left( x_{\varepsilon}\right) \right) -f\left( s,\mathbb{\pi}_{D}\left(
x_{\delta}\right) \right) \right] ds \\
& +Q_{\varepsilon,\delta}^{1/2}\left( s\right) \left( s\right) \left[
-H\left( x_{\varepsilon}\right) \nabla\varphi_{\varepsilon}\left(
x_{\varepsilon}\right) +H\left( x_{\delta}\right) \nabla\varphi_{\delta
}\left( x_{\delta}\right) \right] ds,
\end{align*}
where $\alpha_{\varepsilon,\delta}$, $\beta_{\varepsilon,\delta}:\mathbb{R}%
_{+}\rightarrow\mathbb{R}^{d\times d}$ are some continuous functions which
are bounded uniformly in $\varepsilon,\delta.$

Therefore, for $s\in \left[0,T\right],$%
\begin{align*}
\left\langle u_{\varepsilon,\delta }\left( s\right) ,du_{\varepsilon
,\delta }\left(s\right)\right\rangle & \leq C\left\vert u_{\varepsilon
,\delta }\left(s\right)\right\vert ^{2}\left( d\left\updownarrow
x_{\varepsilon }\right\updownarrow _{s}+d\left\updownarrow x_{\delta
}\right\updownarrow _{s}\right) +C\mu \left( s\right) \left\vert
u_{\varepsilon,\delta }\left( s\right) \right\vert ^{2}ds \\
& +2\left\langle x_{\varepsilon }-x_{\delta },Q_{\varepsilon ,\delta }\left(
s\right) \left[ H\left( x_{\delta }\right) \nabla \varphi _{\delta }\left(
x_{\delta }\right) -H\left( x_{\varepsilon }\right) \nabla \varphi
_{\varepsilon }\left( x_{\varepsilon }\right) \right] \right\rangle ds \\
& \leq C\left\vert u_{\varepsilon ,\delta }\left( s\right) \right\vert
^{2}dV\left( s\right) +4\left\langle x_{\varepsilon }-x_{\delta },\nabla
\varphi _{\delta }\left( x_{\delta }\right) -\nabla \varphi _{\varepsilon
}\left( x_{\varepsilon }\right) \right\rangle ds,
\end{align*}%
with $V\left( s\right) =\left\updownarrow x_{\varepsilon }\right\updownarrow
_{s}+\left\updownarrow x_{\delta }\right\updownarrow _{s}+\left\updownarrow
k_{\varepsilon }\right\updownarrow _{s}+\left\updownarrow k_{\delta
}\right\updownarrow _{s}+\displaystyle\int_{0}^{s}\mu \left( r\right) dr\leq
C_{T}.$

\noindent Since, according to Asiminoaei \& R\u{a}\c{s}canu \cite{Asiminoaei/Rascanu:97},
$$
\left\langle \nabla \varphi _{\varepsilon }(x)-\nabla \varphi _{\delta
}(y),x-y\right\rangle \geq -\left( \varepsilon +\delta \right) |\nabla
\varphi _{\varepsilon }(x)||\nabla \varphi _{\delta }(y)|,
$$
we have
\begin{align*}
\left\langle x_{\varepsilon }\left( r\right) -x_{\delta }\left( r\right)
,dk_{\delta }\left( r\right) -dk_{\varepsilon }\left( r\right) \right\rangle
& =\left\langle x_{\varepsilon }\left( r\right) -x_{\delta }\left( r\right)
,\nabla \varphi \left( x_{\delta }\left( r\right) \right) -\nabla \varphi
\left( x_{\varepsilon }\left( r\right) \right) \right\rangle dr \\
& \leq \left( \varepsilon +\delta \right) \left\vert \nabla \varphi \left(
x_{\delta }\left( r\right) \right) \right\vert \left\vert \nabla \varphi
\left( x_{\varepsilon }\left( r\right) \right) \right\vert dr.
\end{align*}

\noindent Consequently,
$$
\left\langle u_{\varepsilon,\delta }\left( r\right),du_{\varepsilon,\delta }\left( r\right) \right\rangle \leq 4\left( \varepsilon +\delta
\right) \left\vert \nabla \varphi \left( x_{\delta }\left( r\right) \right)
\right\vert \left\vert \nabla \varphi \left(x_{\varepsilon}\left( r\right)
\right) \right\vert dr+C\left\vert u_{\varepsilon,\delta }\left( r\right)
\right\vert ^{2}dV\left(r\right),
$$

\noindent Using inequality (\ref{ineq1-Anex}) from Annex 4.3. we deduce that there exists some positive constants, that will be denoted by a generic one $C$, such that
\begin{align*}
\left\Vert x_{\varepsilon}-x_{\delta}\right\Vert _{T} & \leq C\left\Vert
u_{\varepsilon,\delta}\right\Vert _{T} \\
& \leq C\sqrt{\varepsilon+\delta}\left( \displaystyle\int
_{0}^{T}\left\vert \nabla\varphi\left( x_{\delta}\left( r\right) \right)
\right\vert \left\vert \nabla\varphi\left( x_{\varepsilon}\left( r\right)
\right) \right\vert dr\right) ^{1/2} \\
& \leq C\sqrt{\varepsilon+\delta}\left[ \left( \displaystyle\int
_{0}^{T}\left\vert \nabla\varphi\left( x_{\delta}\left( r\right) \right)
\right\vert ^{2}dr\right) ^{1/2}+\left( \displaystyle\int _{0}^{T}\left\vert
\nabla\varphi\left( x_{\varepsilon}\left( r\right) \right) \right\vert
^{2}dr\right) ^{1/2}\right] \\
& \leq C\sqrt{\varepsilon+\delta}.
\end{align*}

\noindent Now, the other assertions clearly follows and the proof is complete.\hfill
\end{proof}

\begin{corollary}
If $\left(\Omega,\mathcal{F},\mathbb{P},\{\mathcal{F}_{t}\}_{t\geq
0}\right) $ is a stochastic basis and $M$ a $\mathcal{F}_{t}-$progressively
measurable stochastic process such that $M_{\cdot }\left( \omega \right) \in
C^{1}\left( \mathbb{R}_{+};\mathbb{R}^{d}\right),\;\mathbb{P}-a.s.\;\omega
\in \Omega$, then, under the assumptions of Proposition \ref{oSP-p1-uniq}, $%
\mathbb{P}-a.s.\;\omega \in \Omega $, the random generalized Skorohod
problem with oblique subgradients:\newline
$$
\left\{
\begin{array}{r}
X_{t}\left( \omega \right) +\displaystyle\int_{0}^{t}H\left( X_{t}\left(
\omega \right) \right) dK_{t}\left( \omega \right) =x_{0}+\displaystyle%
\int_{0}^{t}f\left( s,X_{s}\left( \omega \right) \right) ds+M_{t}\left(
\omega \right),\quad t\geq 0,\medskip \\
\multicolumn{1}{l}{dK_{t}\left( \omega \right) \in \partial \varphi \left(
X_{t}\left( \omega \right) \right) \left( dt\right)}%
\end{array}%
\right.
$$
admits a unique solution $\left( X_{\cdot }\left( \omega \right) ,K_{\cdot
}\left(\omega\right)\right).$ Moreover $X$ and $K$ are $\mathcal{F}_{t}-$%
progressively measurable stochastic processes.
\end{corollary}

\begin{proof}[{\bf Proof}]
In this moment we have to prove that $X$ and $K$ are $\mathcal{F}_{t}-$%
progressively measurable stochastic processes. But this follows from
Proposition \ref{p-aproxm}, since the approximating equation (\ref{aproxeq})
admits a unique solution $\left( X^{\varepsilon},K^{\varepsilon }\right)$,
which is a progressively measurable continuous stochastic process.\hfill
\end{proof}

\section{SVI with oblique subgradients}

\subsection{Notations. Hypotheses}

In this section we will present the Stochastic Variational Inequalities (for
short, SVI) with oblique subgradient and the definition of theirs strong and
weak solutions. The proof of the existence and uniqueness results are given
in the next subsection.

Let $\left( \Omega ,\mathcal{F},\mathbb{P},\{\mathcal{F}_{t}\}_{t\geq
0}\right) $ be a stochastic basis and $\left\{ B_{t}:t\geq 0\right\} $ a $%
\mathbb{R}^{k}-$valued Brownian motion. Our objective is to solve the SVI
with oblique reflection%
\begin{equation}
\left\{
\begin{array}{l}
X_{t}+\displaystyle\int_{0}^{t}H\left( X_{t}\right) dK_{t}=x_{0}+%
\displaystyle\int_{0}^{t}f\left( s,X_{s}\right) ds+\displaystyle%
\int_{0}^{t}g\left( s,X_{s}\right) dB_{s},\quad t\geq 0,\smallskip \\
dK_{t}\in \partial \varphi \left( X_{t}\right) \left( dt\right),%
\end{array}%
\right.  \label{oSP-eq2}
\end{equation}%
where $x_{0}\in \mathbb{R}^{d}$ and%
\begin{equation}
\begin{array}{rl}
\left( i\right) \quad & \left( t,x\right) \longmapsto f\left( t,x\right) :%
\mathbb{R}_{+}\times \mathbb{R}^{d}\rightarrow \mathbb{R}^{d}\textrm{\ and}%
\;\left( t,x\right) \longmapsto g\left( t,x\right) :\mathbb{R}_{+}\times
\mathbb{R}^{d}\rightarrow \mathbb{R}^{d\times k}\;\textrm{are}\smallskip \\
& \quad \textrm{Carath\'{e}odory functions (i.e. measurable w.r. to }t\textrm{
and continuous w.r. to }x\textrm{),}\medskip \\
\left( ii\right) \quad & \displaystyle\int_{0}^{T}(f^{\#}\left(
t\right)) ^{2}dt+\displaystyle\int_{0}^{T}(g^{\#}\left(t\right)
)^{4}dt<\infty,%
\end{array}
\label{ob-h6}
\end{equation}%
with%
$$
f^{\#}\left( t\right) \overset{def}{=}\sup_{x\in Dom\left( \varphi \right)
}\left\vert f\left( t,x\right) \right\vert \quad \textrm{and}\quad
g^{\#}\left( t\right) \overset{def}{=}\sup_{x\in Dom\left( \varphi \right)
}\left\vert g\left( t,x\right) \right\vert.
$$
We also add Lipschitz continuity conditions:%
\begin{equation}
\begin{array}{rl}
& \exists ~\mu \in L_{loc}^{1}\left( \mathbb{R}_{+}\right) ,\;\;\exists
~\ell \in L_{loc}^{2}\left( \mathbb{R}_{+}\right) \textrm{ s.t. }\forall
~x,y\in \mathbb{R}^{d},\quad a.e.\ t\geq 0,\medskip \\
\left( i\right) \quad & \quad \quad \left\vert f\left( t,x\right) -f\left(
t,y\right) \right\vert \leq \mu \left( t\right) \left\vert x-y\right\vert
,\medskip \\
\left( ii\right) \quad & \quad \quad \left\vert g\left( t,x\right) -g\left(
t,y\right) \right\vert \leq \ell \left( t\right) \left\vert x-y\right\vert .%
\end{array}
\label{ob-h7}
\end{equation}

\begin{definition}
\label{def-weak-strong-sol} $\left( I\right) $ Given a stochastic basis $%
(\Omega ,\mathcal{F},\mathbb{P},\left\{ \mathcal{F}_{t}\right\} _{t\geq 0})$
and a $\mathbb{R}^{k}-$valued $\mathcal{F}_{t}-$Brownian motion $\left\{
B_{t}:t\geq 0\right\} ,$ a pair $\left( X,K\right) :\Omega \times \left[
0,\infty \right[ \rightarrow \mathbb{R}^{d}\times \mathbb{R}^{d}$ of
continuous $\mathcal{F}_{t}-$progressively measurable stochastic processes
is a strong solution of the SDE (\ref{oSP-eq2}) if, $\mathbb{P}-a.s.\;\omega
\in \Omega :$%
\begin{equation}
\left\{
\begin{array}{rl}
i)\; & X_{t}\in \overline{Dom\left( \varphi \right) },\textrm{ \thinspace }%
\forall \,t\geq 0,\;\varphi \left( X_{\cdot }\right) \in L_{loc}^{1}\left(
\mathbb{R}_{+}\right) ,\smallskip \\
ii)\; & K_{\cdot }\in BV_{loc}\left( \left[ 0,\infty \right[ ;\mathbb{R}%
^{d}\right) ,\textrm{\quad\ }K_{0}=0\textrm{,}\smallskip \\
iii)\; & X_{t}+\displaystyle\int_{0}^{t}H\left( X_{s}\right) dK_{s}=x_{0}+%
\displaystyle\int_{0}^{t}f\left( s,X_{s}\right) ds+\displaystyle%
\int_{0}^{t}g\left( s,X_{s}\right) dB_{s},\ \forall ~t\geq 0,\smallskip \\
iv)\; & \forall \,0\leq s\leq t,\;\forall y:\mathbb{R}_{+}\rightarrow
\mathbb{R}^{d}\textrm{ continuous}:\smallskip \\
& \quad \quad \displaystyle\int_{s}^{t}\left\langle y\left( r\right)
-X_{r},dK_{r}\right\rangle +\displaystyle\int_{s}^{t}\varphi \left(
X_{r}\right) dr\leq \displaystyle\int_{s}^{t}\varphi \left( y\left( r\right)
\right) dr.%
\end{array}%
\right.  \label{sp-20a}
\end{equation}%
That is%
$$
\left( X_{\cdot }\left( \omega \right) ,K_{\cdot }\left( \omega \right)
\right) \in \mathcal{SP}\left( H\partial \varphi ;x_{0},M_{\cdot }\left( \omega
\right)\right),\quad \mathbb{P}-a.s.\;\omega \in \Omega,
$$
with%
$$
M_{t}=\displaystyle\int_{0}^{t}f\left( s,X_{s}\right) ds+\displaystyle%
\int_{0}^{t}g\left( s,X_{s}\right) dB_{s}~.
$$
$\left( II\right) \quad $If there exists a stochastic basis $\left( \Omega ,%
\mathcal{F},\mathbb{P},\mathcal{F}_{t}\right)_{t\geq 0}$, a $\mathbb{R}%
^{k}-$valued $\mathcal{F}_{t}-$Brownian motion $\left\{ B_{t}:t\geq
0\right\} $ and a pair $\left( X_{\cdot },K_{\cdot }\right) :\Omega \times
\mathbb{R}_{+}\rightarrow \mathbb{R}^{d}\times \mathbb{R}^{d}$ of $\mathcal{F%
}_{t}-$progressively measurable continuous stochastic processes such that
$$
\left( X_{\cdot}\left(\omega\right),K_{\cdot}\left(\omega\right)
\right) \in \mathcal{SP}\left(H\partial \varphi;x_{0},M_{\cdot}\left(\omega
\right)\right),\quad \mathbb{P}-a.s.\;\omega \in \Omega,
$$
then the collection $\left(\Omega,\mathcal{F},\mathbb{P},\mathcal{F}%
_{t},B_{t},X_{t},K_{t}\right)_{t\geq 0}$ is called a weak solution of the
SVI (\ref{oSP-eq2}).

\noindent (In both cases $\left(I\right)$ and $\left(II\right)$ we will
say that $\left( X_{t},K_{t}\right) $ is a solution of the oblique reflected
SVI (\ref{oSP-eq2}).)
\end{definition}

\subsection{Existence and uniqueness}

In this section we will give the result of existence and uniqueness of the
solution for the stochastic variational inequality with oblique subgradients
introduced before. Theorem \ref{weak-ex-OSVI} deals with the existence of
a weak solution in the sense of Definition \ref{def-weak-strong-sol},
while Theorem \ref{path-w-uni} proves the uniqueness of a strong solution.

\begin{theorem}
\label{weak-ex-OSVI} Let the assumptions (\ref{osp-h0-A}), (\ref{osp-h2}), (%
\ref{osp-h3}) and (\ref{ob-h6}) be satisfied. Then the SVI (\ref{oSP-eq2})
has at least one weak solution $\left( \Omega,\mathcal{F},\mathbb{P},%
\mathcal{F}_{t},B_{t},X_{t},K_{t}\right) _{t\geq0}.$
\end{theorem}

\begin{proof}[{\bf Proof}]
The main ideas of the proof come from Rascanu \cite{Rascanu:10}. We extend $%
f\left( t,x\right) =0$ and $g\left( t,x\right) =0,$ for $t<0$.

{\bf{\it{Step 1.}}{\it Approximating problem.}}$\smallskip $

Let $0<\varepsilon \leq 1$ and consider the approximating equation%
\begin{equation}
\left\{
\begin{array}{l}
X_{t}^{n}=x_{0},\quad \textrm{if }t<0,\medskip \\
X_{t}^{n}+\displaystyle\int_{0}^{t}H\left( X_{t}^{n}\right)
dK_{t}^{n}=x_{0}+M_{t}^{n},\quad t\geq 0,\smallskip \\
dK_{t}^{n}\in \partial \varphi \left( X_{t}^{n}\right) dt,%
\end{array}%
\right.  \label{oea_stoch}
\end{equation}%
where%
\begin{align*}
M_{t}^{n}& =\displaystyle\int_{0}^{t}f(s,\mathbb{\pi }%
_{D}(X_{s-1/n}^{n}))ds+n\displaystyle\int_{t-1/n}^{t}\left[ \displaystyle%
\int_{0}^{s}g(r,\mathbb{\pi }_{D}(X_{r-1/n}^{n}))dB_{r}\right] ds \\
& =\displaystyle\int_{0}^{t}f(s,\mathbb{\pi }_{D}(X_{s-1/n}^{n}))ds+%
\displaystyle\int_{0}^{1}\left[ \displaystyle\int_{0}^{t-\frac{1}{n}+\frac{1%
}{n}u}g(r,\mathbb{\pi }_{D}(X_{r-1/n}^{n}))dB_{r}\right] du
\end{align*}%
and $\mathbb{\pi }_{D}\left( x\right) $ is the orthogonal projection of $x$
on $D=\overline{Dom\left( \varphi \right) }.$ Since $M^{n}$ is a $C^{1}-$%
continuous progressively measurable stochastic process, then by Corollary $%
\mathbf{1}$, the approximating equation (\ref{oea_stoch}) has a unique
solution $\left( X^{n},K^{n}\right) $ of continuous progressively measurable
stochastic processes.$\smallskip $

{\bf{\it{Step 2.}}{\it Tightness.}}$\smallskip $

Let $T\geq 0$ be arbitrary fixed. We will point out the main reasonings of this step.

\begin{itemize}
\item Since, by standard arguments,%
\begin{align*}
& \mathbb{E}\left[ \sup\limits_{0\leq \theta \leq \varepsilon }\left\vert
M_{t+\theta }^{n}-M_{t}^{n}\right\vert ^{4}\right] \\
& \leq 8\left( \displaystyle\int_{t}^{t+\varepsilon }f^{\#}\left( r\right)
dr\right) ^{4}+8\displaystyle\int_{0}^{1}\mathbb{E}\sup_{0\leq \theta \leq
\varepsilon }\left( \displaystyle\int_{t-\frac{1}{n}+\frac{1}{n}u}^{t+\theta
-\frac{1}{n}+\frac{1}{n}u}g(r,\mathbb{\pi }_{D}(X_{r-1/n}^{n}))dBr\right)
^{4}du \\
& \leq 8\varepsilon \left( \displaystyle\int_{t}^{t+\varepsilon
}|f^{\#}\left( r\right) |^{2}dr\right) ^{2}+C\displaystyle\int_{0}^{1}\left( %
\displaystyle\int_{t-\frac{1}{n}+\frac{1}{n}u}^{t+\varepsilon -\frac{1}{n}+%
\frac{1}{n}u}|g^{\#}\left( r\right) |^{2}dr\right) ^{2}du \\
& \leq 8\varepsilon \left( \displaystyle\int_{t}^{t+\varepsilon
}|f^{\#}\left( r\right) |^{2}dr\right) ^{2}+C\varepsilon \displaystyle%
\int_{0}^{1}\left( \displaystyle\int_{t-\frac{1}{n}+\frac{1}{n}%
u}^{t+\varepsilon -\frac{1}{n}+\frac{1}{n}u}|g^{\#}\left( r\right)
|^{4}dr\right) du \\
& \leq C^{\prime }\varepsilon \times \sup \left\{ \left( \displaystyle%
\int_{s}^{\tau }|f^{\#}\left( r\right) |^{2}dr\right) ^{2}+\displaystyle%
\int_{s}^{\tau }|g^{\#}\left( r\right) |^{4}dr;0\leq s<\tau \leq T,\;\;\tau
-s\leq \varepsilon \right\},
\end{align*}
in conformity with Proposition \ref{ch1-p1-tight} the family of laws of $\left\{
M^{n}:n\geq 1\right\} $ is tight on $C\left( \left[ 0,T\right];\mathbb{R}%
^{d}\right)$.

\item We now show that the family of laws of the random variables $%
U^{n}=\left( X^{n},K^{n},\left\updownarrow K^{n}\right\updownarrow \right) $
is tight on $C\left( \left[ 0,T\right] ;\mathbb{R}^{d}\right) \times C\left( %
\left[ 0,T\right] ;\mathbb{R}^{d}\right) \times C\left( \left[ 0,T\right] ;%
\mathbb{R}\right) \left[ =C\left( \left[ 0,T\right] ;\mathbb{R}%
^{2d+1}\right) \right].$ From Proposition \ref{p1-apri-estim} we deduce%
\begin{align*}
\left\Vert U^{n}\right\Vert _{T}& \leq C_{T}\left( \left\Vert
M^{n}\right\Vert _{T}\right), \\
\mathbf{m}_{U^{n}}\left( \varepsilon \right) & \leq C_{T}\left( \left\Vert
M^{n}\right\Vert _{T}\right) \times \sqrt{\varepsilon +\mathbf{m}%
_{M^{n}}\left( \varepsilon \right)},
\end{align*}%
and, from Lemma \ref{ch1-p2-tight}, it follows that $\left\{ U^{n};n\in
\mathbb{N}^{\ast }\right\} $ is tight on $C\left(\left[ 0,T\right];\mathbb{%
R}^{2d+1}\right).$

\item By the Prohorov theorem there exists a subsequence such that, as $%
n\rightarrow \infty,$%
$$
\left( X^{n},K^{n},\left\updownarrow K^{n}\right\updownarrow,B\right)
\rightarrow \left( X,K,V,B\right),\quad \textrm{in law}
$$
on $C\left( \left[ 0,T\right] ;\mathbb{R}^{2d+1+k}\right)$ and, by the
Skorohod theorem, we can choose a probability space $\left( \Omega,\mathcal{%
F},\mathbb{P}\right) $ and some random quadruples $(\bar{X}^{n},\bar{K}^{n},%
\bar{V}^{n},\bar{B}^{n})$, $(\bar{X},\bar{K},\bar{V},\bar{B})$ defined on $%
\left( \Omega ,\mathcal{F},\mathbb{P}\right)$, having the same laws as
resp. $\left( X^{n},K^{n},\left\updownarrow K^{n}\right\updownarrow
,B\right) $ and $(X,K,V,B),$ such that, in $C\left(\left[0,T\right];\mathbb{%
R}^{2d+1+k}\right)$, as $n\rightarrow \infty,$%
$$
(\bar{X}^{n},\bar{K}^{n},\bar{V}^{n},\bar{B}^{n}){\xrightarrow[]{\mathbb{P}-a.s.}}(\bar{X},\bar{K},\bar{V},\bar{B}).
$$

\item Remark that, by Lemma \ref{ch2-p-conv}, $(\bar{B}^{n},\{\mathcal{F}%
_{t}^{\bar{X}^{n},\bar{K}^{n},\bar{V}^{n},\bar{B}^{n}}\}),n\geq 1,\;$and $(%
\bar{B},\{\mathcal{F}_{t}^{\bar{X},\bar{K},\bar{V},\bar{B}}\})$ are $\mathbb{%
R}^{k}-$Brownian motion.
\end{itemize}

{\bf{\it{Step 3.}}{\it Passing to the limit.}}$\smallskip$

Since we have $\left( X^{n},K^{n},\left\updownarrow
K^{n}\right\updownarrow,B\right) \rightarrow (\bar{X},\bar{K},\bar{V},\bar{B})$ in law, then by Proposition \ref{ch1-lsc-SI}, we deduce that, for all $
0\leq s\leq t,$ $\mathbb{P}-a.s.,$
\begin{equation}
\begin{array}{c}
\bar{X}_{0}=x_{0}~,\quad \quad \bar{K}_{0}=0,\quad \quad \bar{X}_{t}\in
E,\medskip \\
\left\updownarrow \bar{K}\right\updownarrow _{t}-\left\updownarrow \bar{K}%
\right\updownarrow _{s}\leq \bar{V}_{t}-\bar{V}_{s}\quad \textrm{and}\quad 0=%
\bar{V}_{0}\leq \bar{V}_{s}\leq \bar{V}_{t}.%
\end{array}
\label{sp-24}
\end{equation}%
Moreover, since for all $0\leq s<t$,$\;n\in \mathbb{N}^{\ast }$%
$$
\int_{s}^{t}\varphi \left( X_{r}^{n}\right) dr\leq \int_{s}^{t}\varphi
\left( y\left( r\right) \right) dr-\int_{s}^{t}\left\langle y\left( r\right)
-X_{r}^{n},dK_{r}^{n}\right\rangle \;\;a.s.,
$$
then, by Proposition \ref{ch1-lsc-SI}, we infer
\begin{equation}
\displaystyle\int_{s}^{t}\varphi \left( \bar{X}_{r}\right) dr\leq %
\displaystyle\int_{s}^{t}\varphi \left( y\left( r\right) \right) dr-%
\displaystyle\int_{s}^{t}\left\langle y\left( r\right) -\bar{X}_{r},d\bar{K}%
_{r}\right\rangle.  \label{sp-25}
\end{equation}%
Hence, based on (\ref{sp-24}) and (\ref{sp-25}), we have
$$d\bar{K}_{r}\in \partial \varphi \left( \bar{X}_{r}\right)\left( dr\right).$$

\noindent Using the Lebesgue theorem and, once again Lemma \ref{ch2-p-conv},
we infer for $n\rightarrow \infty $,%
\begin{align*}
\bar{M}_{\cdot }^{n}& =x_{0}+\displaystyle\int_{0}^{\cdot }f(s,\mathbb{\pi }%
_{D}(\bar{X}_{s-1/n}^{n}))ds+n\displaystyle\int_{\cdot -1/n}^{\cdot }\left[ %
\displaystyle\int_{0}^{s}g(r,\mathbb{\pi }_{D}(\bar{X}_{r-1/n}^{n}))dB_{r},%
\right] ds \\
& \longrightarrow \bar{M}_{\cdot }=x_{0}+\displaystyle\int_{0}^{\cdot }f(s,%
\bar{X}_{s})ds+\displaystyle\int_{0}^{\cdot }g(s,\bar{X}_{s})d\bar{B}%
_{s},\quad \textrm{in }S_{d}^{0}\left[ 0,T\right],
\end{align*}
where $S_{d}^{0}\left[ 0,T\right]$ is the space of progressively measurable continuous stochastic processes defined in Annex, Section 4.3.{\smallskip}

\noindent By Proposition \ref{ch3-c4-cont} it follows that the probability laws equality holds%
$$
\mathcal{L}\left( \bar{X}^{n},\bar{K}^{n},\bar{B}^{n},\bar{M}^{n}\right) =%
\mathcal{L}\left( X^{n},K^{n},B^{n},M^{n}\right) \quad \textrm{on\ }C(\mathbb{R}%
_{+};\mathbb{R}^{d+d+k+d}),
$$
where by $\mathcal{L}(\cdot)$ we mean the probability law of the random variable.

Since, for every $t\geq 0,$%
$$
X_{t}^{n}+\displaystyle\int_{0}^{t}H\left( X_{s}^{n}\right)
dK_{s}^{n}-M_{t}^{n}=0,\;\;a.s.,
$$
then, by Proposition \ref{ch1-lsc-SI}, we have%
$$
\bar{X}_{t}^{n}+\displaystyle\int_{0}^{t}H\left( \bar{X}_{s}^{n}\right) d%
\bar{K}_{s}^{n}-\bar{M}_{t}^{n}=0,\;\;a.s.
$$
Letting $n\rightarrow \infty,$%
$$
\bar{X}_{t}+\displaystyle\int_{0}^{t}H\left( \bar{X}_{s}\right) d\bar{K}_{s}-%
\bar{M}_{t}=0,\;\;a.s.,
$$
that is, $\mathbb{P}-a.s.,$%
$$
\bar{X}_{t}+\displaystyle\int_{0}^{t}H\left(\bar{X}_{s}\right)d\bar{K}%
_{s}=x_{0}+\displaystyle\int_{0}^{t}f\left(s,\bar{X}_{s}\right)ds+%
\displaystyle\int_{0}^{t}g\left(s,\bar{X}_{s}\right) d\bar{B}_{s},\;\forall
~t\in \left[ 0,T\right].
$$

Consequently $(\bar{\Omega},\mathcal{\bar{F}},\mathbb{\bar{P}},\mathcal{F}%
_{t}^{\bar{B},\bar{X}},\bar{X}_{t},\bar{K}_{t},\bar{B}_{t})_{t\geq 0}$ is a
weak solution of the SVI (\ref{oSP-eq2}). The proof is complete.\hfill
\end{proof}

\begin{theorem}
\label{path-w-uni} If the assumptions (\ref{osp-h0-A}), (\ref{osp-h2}), (\ref%
{osp-h3}), (\ref{ob-h6}) and (\ref{ob-h7}) are satisfied, then the SVI (\ref%
{oSP-eq2}) has a unique strong solution $\left( X,K\right) \in
S_{d}^{0}\times S_{d}^{0}.$
\end{theorem}

\begin{proof}[{\bf Proof}]
It is sufficient to prove the {\it pathwise uniqueness}, since by Theorem
1.1, page 149, from Ikeda \& Watanabe \cite{Ikeda/Watanabe:81} {\it the
existence of a weak solution} and {\it the pathwise uniqueness} implies
the existence of a strong solution.\smallskip

Let $\left( X,K\right)$, $(\hat{X},\hat{K})\in S_{d}^{0}\times S_{d}^{0}$
two solutions of the SVI with oblique reflection (\ref{oSP-eq2}). Consider
the symmetric and strict positive matrix
$$
Q_{r}=H^{-1}\left( X_{r}\right) +H^{-1}(\hat{X}_{r}).
$$
We have that%
$$
dQ_{r}^{1/2}=dN_{r}+\displaystyle\sum_{j=1}^{k}\beta _{r}^{\left( j\right)
}dB_{r}^{\left( j\right)},
$$
where $N$ is a $\mathbb{R}^{d\times d}-$valued $\mathcal{P}-$measurabe bounded variation continuous stochastic process
(for short, m.b-v.c.s.p.), $%
N_{0}=0$ and, for each $j\in \overline{1,k}$, $\beta ^{\left( j\right)}$ is
a $\mathbb{R}^{d\times d}-$valued $\mathcal{P}-$measurable stochastic process (for short, m.s.p.) such that $%
\displaystyle\int_{0}^{T}|\beta _{r}^{\left( j\right) }|^{2}dr<\infty$, $%
a.s.,$ for all $T>0.$ \newline
Letting%
$$
U_{r}=Q_{r}^{1/2}(X_{r}-\hat{X}_{r}),
$$
then%
\begin{align*}
dU_{r}& =\left[ dQ_{r}^{1/2}\right] (X_{r}-\hat{X}_{r})+Q_{r}^{1/2}d(X_{r}-%
\hat{X}_{r})+\displaystyle\sum_{j=1}^{k}\beta _{r}^{\left(j\right)
}(g(r,X_{r})-g(r,\hat{X}_{r}))e_{j} \\
& =d\mathcal{K}_{r}+\mathcal{G}_{r}dB_{r},
\end{align*}%
where%
\begin{align*}
d\mathcal{K}_{r}& =\left( dN_{r}\right) Q_{r}^{-1/2}U_{r}+Q_{r}^{1/2}\left[
H(\hat{X}_{r})d\hat{K}_{r}-H\left( X_{r}\right) dK_{r}\right] \\
& +Q_{r}^{1/2}\left[ f\left( r,X_{r}\right) -f(r,\hat{X}_{r})\right] dr+%
\displaystyle\sum_{j=1}^{k}\beta _{r}^{\left( j\right) }(g(r,X_{r})-g(r,\hat{%
X}_{r}))e_{j},\medskip \\
\mathcal{G}_{r}& =\Gamma _{r}+Q_{r}^{1/2}\left[ g(r,X_{r})-g(r,\hat{X}_{r})%
\right],
\end{align*}%
and $\Gamma _{r}$ is a $\mathbb{R}^{d\times k}$ matrix with the columns $%
\beta _{r}^{\left( 1\right) }(X_{r}-\hat{X}_{r})$, \ldots\, $\beta
_{r}^{\left( k\right) }(X_{r}-\hat{X}_{r}).$

Using (\ref{eui_1}) and the properties of $H$ and $H^{-1}$, we have%
\begin{align*}
& \left\langle U_{r},Q_{r}^{1/2}\left[ H(\hat{X}_{r})d\hat{K}_{r}-H\left(
X_{r}\right) dK_{r}\right] \right\rangle \\
& =\left\langle X_{r}-\hat{X}_{r},\left( \left[ H\left( X_{r}\right) \right]
^{-1}-\left[ H(\hat{X}_{r})\right] ^{-1}\right) \left[ H(\hat{X}_{r})d\hat{K}%
_{r}+H\left( X_{r}\right) dK_{r}\right] \right\rangle \\
& -2\left\langle X_{r}-\hat{X}_{r},dK_{r}-d\hat{K}_{r}\right\rangle \\
& \leq bc|X_{r}-\hat{X}_{r}|^{2}(d\left\updownarrow
K\right\updownarrow _{r}+d\updownarrow \!\hat{K}\!\updownarrow _{r}).
\end{align*}

\noindent Hence, there exists a positive constant $C=C(b,c,r_{0})$ such that%
$$
\left\langle U_{r},d\mathcal{K}_{r}\right\rangle +\frac{1}{2}\left\vert
\mathcal{G}_{r}\right\vert ^{2}dt\leq |U_{r}|^{2}dV_{r},
$$
where%
$$
dV_{r}=C\times \left( \mu \left( r\right) dr+\ell ^{2}\left( r\right)
dr+d\left\updownarrow N\right\updownarrow _{r}+d\left\updownarrow
K\right\updownarrow _{r}+d\updownarrow \!\hat{K}\!\updownarrow _{r}\right) +C%
\displaystyle\sum_{j=1}^{k}|\beta _{r}^{\left( j\right) }|^{2}dr.
$$
By Proposition \ref{AnexC-p0-fsi} we infer%
$$
\mathbb{E}\frac{e^{-2V_{s}}\left\vert U_{s}\right\vert ^{2}}{%
1+e^{-2V_{s}}\left\vert U_{s}\right\vert ^{2}}\leq \mathbb{E}\frac{%
e^{-2V_{0}}\left\vert U_{0}\right\vert ^{2}}{1+e^{-2V_{0}}\left\vert
U_{0}\right\vert ^{2}}=0.
$$
Consequently,%
$$
Q_{s}^{1/2}(X_{s}-\hat{X}_{s})=U_{s}=0,\;\mathbb{P}-a.s.,\textrm{ for all }%
s\geq 0
$$
and, by the continuity of $X$ and $\hat{X},$ we conclude that, $\mathbb{P}%
-a.s.,$%
$$
X_{s}=\hat{X}_{s}\quad \textrm{for all\ }s\geq 0.
$$
\hfill
\end{proof}

\section{Annex}

For the clarity of the proofs from the main body of this article we will group in this section some useful results that are used along this paper.

\subsection{A priori estimates}

We give five lemmas with a priori estimates of the solutions $\left(
x,k\right) \in \mathcal{SP}\left( H\partial \varphi ;x_{0},m\right).$ These
lemmas and also theirs proofs are similar with those from the monograph of Pardoux $\&$ R\u{a}\c{s}canu
\cite{Pardoux/Rascanu:09}, but for the convenience of the reader we give here the proofs of the results
in this new framework.$%
\smallskip $

\begin{lemma}
If $\left( x,k\right) \in \mathcal{SP}\left( H\partial\varphi;x_{0},m\right) $
and $(\hat{x},\hat{k}) \in \mathcal{SP}\left( H\partial\varphi;\hat{x}_{0},\hat{m%
}\right), $ then for all $0\leq$ $s\leq t:$%
\begin{equation}
\displaystyle\int _{s}^{t}\left\langle x\left( r\right) -\hat{x}\left(
r\right),dk\left( r\right) -d\hat{k}\left( r\right) \right\rangle \geq0.
\label{osp-2}
\end{equation}
\end{lemma}

\noindent We recall the notation for modulus of continuity of a function $g:%
\left[ 0,T\right] \rightarrow \mathbb{R}^{d}:$%
$$
\mathbf{m}_{g}\left( \varepsilon \right) =\sup \left\{ \left\vert g\left(
u\right)-g\left( v\right) \right\vert :u,v\in \left[0,T\right],\;\left\vert u-v\right\vert \leq \varepsilon \right\}.
$$

\begin{lemma}
\label{l1-mc-x}Let the assumptions (\ref{osp-h0}), (\ref{osp-h0-A}), (\ref%
{osp-h2}) and (\ref{osp-h3}) be satisfied. If $\left( x,k\right) \in\mathcal{%
SP}\left( H\partial\varphi;x_{0},m\right) ,$ then for all $0\leq s\leq t\leq
T:$%
\begin{equation}
\begin{array}{r}
\mathbf{m}_{x}\left( t-s\right) \leq\left[ \left( t-s\right) +\mathbf{m}%
_{m}\left( t-s\right) +~\sqrt{\mathbf{m}_{m}\left( t-s\right) \left(
\left\updownarrow k\right\updownarrow _{t}-\left\updownarrow
k\right\updownarrow _{s}\right) }\right] \medskip \\
\times\exp\left\{ C\left[ 1+\left( t-s\right) +\left( \left\updownarrow
k\right\updownarrow _{t}-\left\updownarrow k\right\updownarrow _{s}+1\right)
\left( \left\updownarrow k\right\updownarrow _{t}-\left\updownarrow
k\right\updownarrow _{s}\right) \right] \right\},%
\end{array}
\label{ob1}
\end{equation}
where $C=C\left( b,c,L\right) >0.$
\end{lemma}

\begin{proof}[{\bf Proof.}]
Let $0\leq s\leq t$ and%
$$
h\left(t\right) =\left\langle H^{-1}\left( x\left( t\right) \right) \left[
x\left(t\right) -m\left( t\right) -x\left( s\right) +m\left( s\right) %
\right],x\left( t\right) -m\left( t\right) -x\left( s\right) +m\left(
s\right)\right\rangle.
$$
We have%
\begin{align*}
h\left( t\right) & =2\displaystyle\int_{s}^{t}\left\langle H^{-1}\left(
x\left( t\right) \right) \left[ x\left( r\right) -m\left( r\right) -x\left(
s\right) +m\left( s\right) \right] ,d\left[ x\left( r\right) -m\left(
r\right) -x\left( s\right) +m\left( s\right) \right] \right\rangle  \\
& =-2\displaystyle\int_{s}^{t}\left\langle H^{-1}\left( x\left( t\right)
\right) \left[ x\left( r\right) -m\left( r\right) -x\left( s\right) +m\left(
s\right) \right] ,H\left( x\left( r\right) \right) dk\left( r\right)
\right\rangle  \\
& =2\displaystyle\int_{s}^{t}\left\langle H^{-1}\left( x\left( t\right)
\right) \left[ m\left( r\right) -m\left( s\right) \right] ,H\left( x\left(
r\right) \right) dk\left( r\right) \right\rangle +2\displaystyle%
\int_{s}^{t}\left\langle x\left( s\right) -x\left( r\right) ,dk\left(
r\right) \right\rangle  \\
& +2\displaystyle\int_{s}^{t}\left\langle \left[ H^{-1}\left( x\left(
r\right) \right) -H^{-1}\left( x\left( t\right) \right) \right] \left[
x\left( r\right) -x\left( s\right) \right] ,H\left( x\left( r\right) \right)
dk\left( r\right) \right\rangle .
\end{align*}%
Since%
\begin{align*}
\displaystyle\int_{s}^{t}\left\langle x\left( s\right) -x\left( r\right)
,dk\left( r\right) \right\rangle & \leq \displaystyle\int_{s}^{t}\left[
\varphi \left( x\left( s\right) \right) -\varphi \left( x\left( r\right)
\right) \right] dr \\
& \leq L\left( t-s\right) +L\displaystyle\int_{s}^{t}\left\vert x\left(
s\right) -x\left( r\right) \right\vert dr \\
& \leq L\left( t-s\right) +\frac{L}{2}\left( t-s\right) +\frac{L}{2}%
\displaystyle\int_{s}^{t}\left\vert x\left( r\right) -x\left( s\right)
\right\vert ^{2}dr
\end{align*}%
and%
$$
\frac{1}{2c}\left\vert x\left( t\right) -x\left( s\right) \right\vert ^{2}-%
\frac{1}{c}\left\vert m\left( t\right) -m\left( s\right) \right\vert
^{2}\leq h\left( t\right),
$$
then%
\begin{align*}
\left\vert x\left( t\right) -x\left( s\right) \right\vert ^{2}& \leq 2~%
\mathbf{m}_{m}^{2}\left( t-s\right) +4c^{3}~\mathbf{m}_{m}\left( t-s\right)
\left( \left\updownarrow k\right\updownarrow _{t}-\left\updownarrow
k\right\updownarrow _{s}\right) +6cL\left( t-s\right)  \\
& +\displaystyle\int_{s}^{t}\left[ 2cL\left\vert x\left( r\right) -x\left(
s\right) \right\vert ^{2}dr+4bc^{2}\left\vert x\left( r\right) -x\left(
t\right) \right\vert \left\vert x\left( r\right) -x\left( s\right)
\right\vert d\left\updownarrow k\right\updownarrow _{r}\right] .
\end{align*}%
Here we continue the estimates by
\begin{align*}
& 4bc^{2}\displaystyle\int_{s}^{t}\left\vert x\left( r\right) -x\left(
t\right) \right\vert \left\vert x\left( r\right) -x\left( s\right)
\right\vert d\left\updownarrow k\right\updownarrow _{r} \\
& \leq 4bc^{2}\left\vert x\left( s\right) -x\left( t\right) \right\vert %
\displaystyle\int_{s}^{t}\left\vert x\left( r\right) -x\left( s\right)
\right\vert d\left\updownarrow k\right\updownarrow _{r}+4bc^{2}\displaystyle%
\int_{s}^{t}\left\vert x\left( r\right) -x\left( s\right) \right\vert
^{2}d\left\updownarrow k\right\updownarrow _{r} \\
& \leq \frac{1}{2}\left\vert x\left( s\right) -x\left( t\right) \right\vert
^{2}+\frac{1}{2}\left( 4bc^{2}\right) ^{2}\left( \displaystyle%
\int_{s}^{t}\left\vert x\left( r\right) -x\left( s\right) \right\vert
d\left\updownarrow k\right\updownarrow _{r}\right) ^{2} \\
& +4bc^{2}\displaystyle\int_{s}^{t}\left\vert x\left( r\right) -x\left(
s\right) \right\vert ^{2}d\left\updownarrow k\right\updownarrow _{r}
\end{align*}%
and we obtain%
\begin{align*}
& \left\vert x\left( t\right) -x\left( s\right) \right\vert ^{2} \\
& \leq 4~\mathbf{m}_{m}^{2}\left( t-s\right) +8c^{3}~\mathbf{m}_{m}\left(
t-s\right) \left( \left\updownarrow k\right\updownarrow
_{t}-\left\updownarrow k\right\updownarrow _{s}\right) +12cL\left(
t-s\right)  \\
& +4cL\displaystyle\int_{s}^{t}\left\vert x\left( r\right) -x\left( s\right)
\right\vert ^{2}dr+\left[ 16b^{2}c^{4}\left( \left\updownarrow
k\right\updownarrow _{t}-\left\updownarrow k\right\updownarrow _{s}\right)
+8bc^{2}\right] \displaystyle\int_{s}^{t}\left\vert x\left( r\right)
-x\left( s\right) \right\vert ^{2}d\left\updownarrow k\right\updownarrow
_{r}~.
\end{align*}%
By the Stieltjes-Gronwall inequality, from this last inequality, the
estimate (\ref{ob1}) follows.\hfill
\end{proof}

For the next result we first remark that, if $E\subset \mathbb{R}^{d}$ is a closed convex set such that%
$$
\exists \textrm{~}r_{0}>0,\;E_{r_{0}}\neq \emptyset \quad \textrm{and}\quad
h_{0}=\sup_{z\in E}dist\left( z,E_{r_{0}}\right) <\infty
$$
(in particular if $E$ is bounded), then for every $0<\delta \leq \frac{r_{0}%
}{2\left( 1+h_{0}\right)},$ $y\in E,$ $\hat{y}=\pi_{E_{r_{0}}}\left(
y\right),$ $v_{y}=\frac{1}{1+h_{0}}\left( \hat{y}-y\right)$ and for all $%
x\in E\cap \overline{B}\left(y,\delta \right)$ we have%
\begin{equation}
\overline{B}\left( x+v_{y},\delta \right) \subset \overline{B}\left( y+v_{y},%
\frac{r_{0}}{1+h_{0}}\right) \subset conv\left\{y,\overline{B}\left(\hat{y}%
,r_{0}\right) \right\} \subset E.  \label{osp-suibc}
\end{equation}

\begin{lemma}
\label{l1-oSp}Let the assumptions (\ref{osp-h0}), (\ref{osp-h0-A}), (\ref%
{osp-h2}) and (\ref{osp-h3}) be satisfied. If $\left( x,k\right) \in
\mathcal{SP}\left( H\partial \varphi ;x_{0},m\right) ,$ $0\leq s\leq t\leq T$
and%
$$
\sup_{r\in \left[ s,t\right] }\left\vert x\left( r\right) -x\left( s\right)
\right\vert \leq 2\delta _{0}=\frac{\rho _{0}}{2bc}\wedge \rho _{0}~,\quad
\textrm{with }\rho _{0}=\frac{r_{0}}{2\left( 1+r_{0}+h_{0}\right) },
$$
then%
\begin{equation}
\left\updownarrow k\right\updownarrow _{t}-\left\updownarrow
k\right\updownarrow _{s}\leq \frac{1}{\rho _{0}}\left\vert k\left( t\right)
-k\left( s\right) \right\vert +\frac{3L}{\rho _{0}}\left( t-s\right)
\label{osp-3}
\end{equation}%
and%
\begin{equation}
\left\vert x\left( t\right) -x\left( s\right) \right\vert +\left\updownarrow
k\right\updownarrow _{t}-\left\updownarrow k\right\updownarrow _{s}\leq ~%
\sqrt{t-s+\mathbf{m}_{m}\left( t-s\right)}\times e^{C_{T}\left(
1+\left\Vert m\right\Vert _{T}^{2}\right)},  \label{ob2}
\end{equation}%
where $C_{T}=C\left( b,c,r_{0},h_{0},L,T\right) >0.$
\end{lemma}

\begin{proof}
Remark first that $D_{r_{0}}\subset D_{\delta _{0}}~.$ Let $\alpha \in
C\left( \left[ 0,\infty \right[ ;\mathbb{R}^{d}\right),$ $\left\Vert \alpha
\right\Vert _{\left[ s,t\right] }\leq 1,$ be arbitrary. Consider $y=x\left(
s\right) \in D,$ $\hat{y}=\pi_{D_{r_{0}}}\left( y\right) $ and%
$$
v_{y}=\dfrac{1}{1+h_{0}}\left( \hat{y}-y\right).
$$

Let $z\left( r\right) =x\left( r\right) +v_{y}+\rho _{0}\alpha \left(
r\right) ,$ $r\in \left[ s,t\right] $. Since $\left\vert x\left( r\right)
-y\right\vert \leq 2\delta _{0}\leq \rho _{0},$ then%
\begin{align*}
x\left( r\right) +v_{y}+\rho _{0}\alpha \left( r\right) & \in \overline{B}%
\left( x\left( r\right) +v_{y},\rho _{0}\right) \\
& \subset \overline{B}\left( y+v_{y},\frac{r_{0}}{1+h_{0}}\right) \\
& \subset D.
\end{align*}%
Remark that%
$$
\left\vert z\left( r\right) -x\left( r\right) \right\vert \leq \frac{h_{0}}{%
1+h_{0}}+\rho _{0}\leq 2
$$
and%
$$
\left\vert \varphi \left( z\left( r\right) \right) -\varphi \left( x\left(
r\right) \right) \right\vert \leq 3L.
$$
Therefore%
\begin{align*}
\rho _{0}\displaystyle\int_{s}^{t}\left\langle \alpha \left( r\right)
,dk\left( r\right) \right\rangle & \leq -\displaystyle\int_{s}^{t}\left%
\langle v_{y},dk\left( r\right) \right\rangle +\displaystyle\int_{s}^{t}%
\left[ \varphi \left( z\left( r\right) \right) -\varphi \left( x\left(
r\right) \right) \right] dr \\
& \leq -\left\langle v_{y},k\left( t\right) -k\left( s\right) \right\rangle
+3L\left( t-s\right) .
\end{align*}%
Taking the $\sup_{\left\Vert \alpha \right\Vert _{\left[ s,t\right] }\leq 1}$,
we infer%
$$
\rho _{0}\left( \left\updownarrow k\right\updownarrow _{t}-\left\updownarrow
k\right\updownarrow _{s}\right) \leq \left\vert k\left( t\right) -k\left(
s\right) \right\vert +3L\left( t-s\right) ,
$$
that is (\ref{osp-3}).

We have also%
\begin{align*}
& \left\updownarrow k\right\updownarrow _{t}-\left\updownarrow
k\right\updownarrow _{s} \\
& \leq \frac{1}{\rho _{0}}\left\vert k\left( t\right) -k\left( s\right)
\right\vert +\frac{3L}{\rho_{0}}\left(t-s\right) \\
& =\frac{1}{\rho _{0}}\displaystyle\int_{s}^{t}\left[ H^{-1}\left( x\left(
r\right) \right) -H^{-1}\left( x\left( s\right) \right) \right] H\left(
x\left( r\right) \right) dk\left( r\right) +\frac{1}{\rho _{0}}H^{-1}\left(
x\left( s\right) \right) \displaystyle\int_{s}^{t}H\left( x\left( r\right)
\right) dk\left( r\right) \\
& +\frac{3L}{\rho _{0}}\left( t-s\right) \\
& \leq \frac{bc}{\rho _{0}}\displaystyle\int_{s}^{t}\left\vert x\left(
r\right) -x\left( s\right) \right\vert d\left\updownarrow
k\right\updownarrow _{r}+\frac{c}{\rho _{0}}\left\vert -x\left( t\right)
+x\left( s\right) +m\left( t\right) -m\left( s\right) \right\vert +\frac{3L}{%
\rho _{0}}\left( t-s\right) \\
& \leq \frac{bc}{\rho _{0}}2\delta _{0}\left( \left\updownarrow
k\right\updownarrow _{t}-\left\updownarrow k\right\updownarrow _{s}\right) +%
\frac{c}{\rho _{0}}\left\vert x\left( t\right) -x\left( s\right) \right\vert
+\frac{c}{\rho _{0}}\mathbf{m}_{m}\left( t-s\right) +\frac{3L}{\rho _{0}}%
\left( t-s\right) \\
& \leq \frac{1}{2}\left( \left\updownarrow k\right\updownarrow
_{t}-\left\updownarrow k\right\updownarrow _{s}\right) +\frac{c}{\rho _{0}}%
\left\vert x\left( t\right) -x\left( s\right) \right\vert +\frac{c}{\rho _{0}%
}\mathbf{m}_{m}\left( t-s\right) +\frac{3L}{\rho _{0}}\left( t-s\right)
\end{align*}%
and, consequently,%
\begin{equation}
\begin{array}{lll}
\left\updownarrow k\right\updownarrow _{t}-\left\updownarrow
k\right\updownarrow _{s} & \leq & \dfrac{2c}{\rho _{0}}\left\vert x\left(
t\right) -x\left( s\right) \right\vert +\dfrac{2c}{\rho _{0}}\mathbf{m}%
_{m}\left( t-s\right) +\dfrac{6L}{\rho _{0}}\left( t-s\right) \medskip \\
& \leq & \dfrac{1}{b}+\dfrac{2c}{\rho _{0}}\mathbf{m}_{m}\left( t-s\right) +%
\dfrac{6L}{\rho _{0}}T\medskip \\
& \leq & C_{1}\left( 1+\left\Vert m\right\Vert _{T}\right) ,%
\end{array}
\label{o4}
\end{equation}%
with $C_{1}=C_{1}\left( T,b,c,\rho _{0},L\right) .$

Now, plugging this estimate in (\ref{ob1}), it clearly follows
$$
\mathbf{m}_{x}\left( t-s\right) \leq \left[ \left( t-s\right) +\mathbf{m}%
_{m}\left( t-s\right) +~\sqrt{\mathbf{m}_{m}\left( t-s\right) }\right] \exp %
\left[ C^{\prime }(1+\left\Vert m\right\Vert _{T}^{2})\right] ,
$$
with $C^{\prime }=C^{\prime }\left( b,c,L,r_{0},h_{0},T\right) .$ Now, this
last inequality, used in (\ref{o4}), yields the estimate (\ref{ob2}).\hfill
\end{proof}

\begin{lemma}
\label{sp-l2}Let the assumptions (\ref{osp-h0}), (\ref{osp-h0-A}), (\ref%
{osp-h2}) and (\ref{osp-h3}) be satisfied. Let $\left( x,k\right) \in%
\mathcal{SP}\left( H\partial\varphi;x_{0},m\right) ,$ $0\leq s\leq t\leq T$
and $x\left( r\right) \in D_{\delta_{0}}$ for all $r\in\left[ s,t\right] $.
Then%
$$
\left\updownarrow k\right\updownarrow _{t}-\left\updownarrow
k\right\updownarrow _{s}\leq L\left( 1+\frac{2}{\delta_{0}}\right)\left(
t-s\right)
$$
and%
$$
\mathbf{m}_{x}\left( t-s\right) \leq C_{T}\times\left[ \left( t-s\right) +%
\mathbf{m}_{m}\left( t-s\right) \right],
$$
where $C_{T}=C_{T}\left( b,c,r_{0},h_{0},L,T\right) >0.$
\end{lemma}

\begin{proof}
Let $y\left( r\right) =x\left( r\right) +\dfrac{\delta _{0}}{2}\alpha \left(
r\right) $, with $\alpha \in C\left( \mathbb{R}_{+};\mathbb{R}^{d}\right) ,$ $%
\left\Vert \alpha \right\Vert _{\left[ s,t\right] }\leq 1.$ Then $y\left(
r\right) \in D$ and
\begin{align*}
\dfrac{\delta _{0}}{2}\displaystyle\int_{s}^{t}\left\langle \alpha \left(
r\right) ,dk\left( r\right) \right\rangle & =\displaystyle%
\int_{s}^{t}\left\langle y\left( r\right) -x\left( r\right) ,dk\left(
r\right) \right\rangle  \\
& \leq \displaystyle\int_{s}^{t}\left[ \varphi \left( y\left( r\right)
\right) -\varphi \left( x\left( r\right) \right) \right] dr \\
& \leq L\left( t-s\right) +L\displaystyle\int_{s}^{t}\left\vert y\left(
r\right) -x\left( r\right) \right\vert dr \\
& \leq L\left( t-s\right) +L\dfrac{\delta _{0}}{2}\left( t-s\right) .
\end{align*}%
Taking the supremum over all $\alpha $ such that $\left\Vert \alpha
\right\Vert _{\left[ s,t\right] }\leq 1$, we have
$$
\left\updownarrow k\right\updownarrow _{t}-\left\updownarrow
k\right\updownarrow _{s} \leq \left( \frac{2L}{\delta _{0}}+L\right)
\left( t-s\right)
$$
and, by Lemma \ref{l1-mc-x}, the result follows.\hfill
\end{proof}

Denote now $\mathbf{\mu}_{m}\left( \varepsilon\right) =\varepsilon +\mathbf{m}%
_{m}\left( \varepsilon\right) ,$ $\varepsilon\geq0.$

\begin{lemma}
\label{oSP-l4-compact}Let the assumptions (\ref{osp-h0}), (\ref{osp-h0-A}), (%
\ref{osp-h2}) and (\ref{osp-h3}) be satisfied and $\left( x,k\right) \in
\mathcal{SP}\left( H\partial \varphi ;x_{0},m\right) .$ Then, there exists a
positive constant $C_{T}\left( \left\Vert m\right\Vert _{T}\right) =C\left(
x_{0},b,c,r_{0},h_{0},L,T,\left\Vert m\right\Vert _{T}\right) ,$ increasing
function with respect to $\left\Vert m\right\Vert _{T},$ such that, for all $%
0\leq s\leq t\leq T:$%
\begin{equation}
\begin{array}{ll}
\left( a\right) \quad & \left\Vert x\right\Vert _{T}+\left\updownarrow
k\right\updownarrow _{T}\leq C_{T}\left( \left\Vert m\right\Vert _{T}\right)
,\medskip \\
\left( b\right) \quad & \left\vert x\left( t\right) -x\left( s\right)
\right\vert +\left\updownarrow k\right\updownarrow _{t}-\left\updownarrow
k\right\updownarrow _{s}\leq C_{T}\left( \left\Vert m\right\Vert _{T}\right)
\times \sqrt{\mathbf{\mu }_{m}\left( t-s\right) }.%
\end{array}
\label{ob4}
\end{equation}
\end{lemma}

\begin{proof}
We will follow the ideas of Lions and Sznitman from \cite{Lions/Sznitman-84}.

{\bf{\it{Step 1.}}} Define the sequence%
\begin{align*}
t_{0}& =T_{0}=0 \\
T_{1}& =\inf \left\{ t\in \left[ t_{0},T\right] :dist\left( x\left( t\right)
,Bd(D)\right) \leq \frac{\delta _{0}}{2}\right\} , \\
t_{1}& =\inf \left\{ t\in \left[ T_{1},T\right] :\left\vert x\left( t\right)
-x\left( T_{1}\right) \right\vert >\delta _{0}\right\} , \\
T_{2}& =\inf \left\{ t\in \left[ t_{1},T\right] :dist\left( x\left( t\right)
,Bd(D)\right) \leq \frac{\delta _{0}}{2}\right\} , \\
& \cdots \;\cdots \;\cdots \\
t_{i}& =\inf \left\{ t\in \left[ T_{i},T\right] :\left\vert x\left( t\right)
-x\left( T_{i}\right) \right\vert >\delta _{0}\right\} , \\
T_{i+1}& =\inf \left\{ t\in \left[ t_{i},T\right] :dist\left( x\left(
t\right) ,Bd(D)\right) \leq \frac{\delta _{0}}{2}\right\} , \\
& \cdots \;\cdots \;\cdots
\end{align*}%
Clearly, we have%
$$
0=T_{0}=t_{0}\leq T_{1}<t_{1}\leq T_{2}<\cdots \leq T_{i}<t_{i}\leq
T_{i+1}<t_{i+1}\leq \cdots \leq T.
$$
Denote%
$$
K\left( t\right) ={\int_{0}^{t}}H\left( x(r)\right) dk\left( r\right)
$$
and it follows that there exists a positive constant $\tilde{c}$ such that
$$
\left\updownarrow K\right\updownarrow _{t}\leq c\left\updownarrow
k\right\updownarrow _{t}\leq \tilde{c}\left\updownarrow K\right\updownarrow
_{t}.
$$

We have

\begin{itemize}
\item for $t_{i}\leq s\leq t\leq T_{i+1}:$%
$$
\left\vert x\left( t\right) -x\left( s\right) \right\vert \leq
\left\updownarrow K\right\updownarrow _{t}-\left\updownarrow
K\right\updownarrow _{s}+\left\vert m\left( t\right) -m\left( s\right)
\right\vert .
$$
Since for $t_{i}\leq r\leq T_{i+1}$, $x\left( r\right) \in D_{\delta _{0}}$
then, by Lemma \ref{sp-l2}, for $t_{i}\leq s\leq t\leq T_{i+1}\,,$%
$$
\left\updownarrow k\right\updownarrow _{t}-\left\updownarrow
k\right\updownarrow _{s}\leq L\left( 1+\frac{2}{\delta _{0}}%
\right) \left( t-s\right)
$$
and%
$$
\mathbf{m}_{x}\left( t-s\right) \leq \left[ \left( t-s\right) +\mathbf{m}%
_{m}\left( t-s\right) \right] \times C_{T}.
$$
Hence, denoting in what follows by $C_{T}(\left\Vert m\right\Vert _{T})$ a generic constant depending on the supremum norm of the continuous function $m$, we have
\begin{align*}
\mathbf{m}_{x}\left( t-s\right) +\left\updownarrow k\right\updownarrow
_{t}-\left\updownarrow k\right\updownarrow _{s}& \leq \mathbf{\mu }%
_{m}\left( t-s\right) \times C_{T} \\
& \leq \sqrt{\mathbf{\mu }_{m}\left( t-s\right) }\times C_{T}\left(
\left\Vert m\right\Vert _{T}\right).
\end{align*}

\item for $T_{i}\leq s\leq t\leq t_{i},$ by Lemma \ref{l1-oSp} we have%
$$
\left\vert x\left( t\right) -x\left( s\right) \right\vert +\left\updownarrow
k\right\updownarrow _{t}-\left\updownarrow k\right\updownarrow _{s}\leq\sqrt{%
\mathbf{\mu}_{m}\left( t-s\right) }\times C_{T}\left( \left\Vert
m\right\Vert _{T}\right).
$$

\item for $T_{i}\leq s\leq t_{i}\leq t\leq T_{i+1}~$,%
\begin{align*}
& \left\vert x\left( t\right) -x\left( s\right) \right\vert
+\left\updownarrow k\right\updownarrow _{t}-\left\updownarrow
k\right\updownarrow _{s} \\
& \leq \left\vert x\left( t\right) -x\left( t_{i}\right) \right\vert
+\left\updownarrow k\right\updownarrow _{t}-\left\updownarrow
k\right\updownarrow _{t_{i}}+\left\vert x\left( t_{i}\right) -x\left(
s\right) \right\vert +\left\updownarrow k\right\updownarrow
_{t_{i}}-\left\updownarrow k\right\updownarrow _{s} \\
& \leq \sqrt{\mathbf{\mu }_{m}\left( t-t_{i}\right) }\times C_{T}+\sqrt{%
\mathbf{\mu }_{m}\left( t_{i}-s\right) }\times C_{T}\left( \left\Vert
m\right\Vert _{T}\right) \\
& \leq \sqrt{\mathbf{\mu }_{m}\left( t-s\right) }\times C_{T}\left( \left\Vert m\right\Vert _{T}\right).
\end{align*}
\end{itemize}

Consequently, for all $i\in \mathbb{N}$ and $T_{i}\leq s\leq t\leq T_{i+1}~,$%
$$
\left\vert x\left( t\right) -x\left( s\right) \right\vert +\left\updownarrow
k\right\updownarrow _{t}-\left\updownarrow k\right\updownarrow _{s}\leq
\sqrt{\mathbf{\mu }_{m}\left( t-s\right) }\times C_{T}\left( \left\Vert
m\right\Vert _{T}\right),
$$
where $C_{T}\left( \left\Vert m\right\Vert _{T}\right) =C\left(
b,c,r_{0},h_{0},L,\left\Vert m\right\Vert _{T}\right) $ is increasing with respect to $%
\left\Vert m\right\Vert _{T}.\smallskip $

{\bf{\it{Step 2.}}} Since $\mathbf{\mu }_{m}:\left[ 0,T%
\right] \rightarrow \left[ 0,\mathbf{\mu }_{m}\left( T\right) \right] $ is a
strictly increasing continuous function, then the inverse function $\mathbf{%
\mu }_{m}^{-1}:\left[ 0,\mathbf{\mu }_{m}\left( T\right) \right] \rightarrow %
\left[ 0,T\right] $ is well defined and it is, also, a strictly increasing
continuous function. We have%
\begin{align*}
\delta & \leq \left\vert x\left( t_{i}\right) -x\left( T_{i}\right)
\right\vert \\
& \leq \sqrt{\mathbf{\mu }_{m}\left( t_{i}-T_{i}\right) }\times C_{T}\left(
\left\Vert m\right\Vert _{T}\right) \\
& \leq \sqrt{\mathbf{\mu }_{m}\left( T_{i+1}-T_{i}\right) }\times
C_{T}\left( \left\Vert m\right\Vert _{T}\right)
\end{align*}%
and, consequently,%
$$
T_{i+1}-T_{i}\geq \mathbf{\mu }_{m}^{-1}\left[ \left( \frac{\delta }{%
C_{T}\left( \left\Vert m\right\Vert _{T}\right) }\right) ^{2}\right] \overset%
{def}{=}\frac{1}{\Delta _{m}}>0.
$$
Therefore, the bounded increasing sequence $\left( T_{i}\right) _{i\geq 0}$ is
finite.\newline
Considering $j$ be such that $T=T_{j},$ we have%
$$
T=T_{j}=\sum_{i=1}^{j}\left( T_{i}-T_{i-1}\right) \geq \frac{j}{\Delta _{m}}%
~.
$$
Let $0\leq s\leq t\leq T$ and we have
\begin{align*}
\left\updownarrow k\right\updownarrow _{t}-\left\updownarrow
k\right\updownarrow _{s}& =\sum_{i=1}^{j}\left( \left\updownarrow
k\right\updownarrow _{\left( t\wedge T_{i}\right) \vee s}-\left\updownarrow
k\right\updownarrow _{\left( t\wedge T_{i-1}\right) \vee s}\right) \\
& \leq \sum_{i=1}^{j}\sqrt{\mathbf{\mu }_{m}\Big(\left( t\wedge T_{i}\right)
\vee s-\left( t\wedge T_{i-1}\right) \vee s\Big)}\times C_{T}\left(
\left\Vert m\right\Vert _{T}\right) \\
& \leq j\times \sqrt{\mathbf{\mu }_{m}\left( t-s\right) }\times C_{T}\left(
\left\Vert m\right\Vert _{T}\right) \\
& \leq T~\Delta _{m}~\sqrt{\mathbf{\mu }_{m}\left( t-s\right) }\times
C_{T}\left( \left\Vert m\right\Vert _{T}\right).
\end{align*}%
Consequently,%
$$
\left\updownarrow k\right\updownarrow _{T}\leq T~\Delta _{m}~\sqrt{\mathbf{%
\mu }_{m}\left( T\right) }\times C_{T}\left( \left\Vert m\right\Vert
_{T}\right) \leq C_{T}^{\prime }\left( \left\Vert m\right\Vert _{T}\right)
$$
and
\begin{align*}
\left\vert x\left( t\right) \right\vert & =\left\vert x_{0}+m\left( t\right)
-\displaystyle\int_{0}^{t}H\left( x\left( s\right) \right) dk\left( s\right)
\right\vert \\
& \leq \left\vert x_{0}\right\vert +\left\Vert m\right\Vert
_{t}+c\left\updownarrow k\right\updownarrow _{t} \\
& \leq \left\vert x_{0}\right\vert +\left\Vert m\right\Vert
_{T}+c\left\updownarrow k\right\updownarrow _{T}.
\end{align*}%
We conclude that there exists a positive constant $C_{T}\left( \left\Vert m\right\Vert
_{T}\right) =C\left( b,c,r_{0},h_{0},L,\left\Vert m\right\Vert _{T}\right)
>0 $ (increasing with respect to $\left\Vert m\right\Vert _{T}$) such that%
$$
\left\updownarrow k\right\updownarrow _{T}\leq C_{T}\left( \left\Vert
m\right\Vert _{T}\right) \;\;\textrm{and \ \ }\left\Vert x\right\Vert _{T}\leq
\left\vert x_{0}\right\vert +C_{T}\left( \left\Vert m\right\Vert _{T}\right),
$$
that is (\ref{ob4}-a).$\smallskip $

By Lemma \ref{l1-mc-x}, for every $\,0\leq s\leq t\leq T:$%
\begin{align*}
\mathbf{m}_{x}\left( t-s\right) & \leq \left[ \left( t-s\right) +\mathbf{m}%
_{m}\left( t-s\right) +~\sqrt{\mathbf{m}_{m}\left( t-s\right) C_{T}\left(
\left\Vert m\right\Vert _{T}\right) }\right] \times C_{T}\left( \left\Vert
m\right\Vert _{T}\right)  \\
& \leq C_{T}^{\prime }\left( \left\Vert m\right\Vert _{T}\right) \times
\sqrt{\mathbf{\mu }_{m}\left( t-s\right) },
\end{align*}%
that means (\ref{ob4}-b) holds. The proof is now complete.\hfill
\end{proof}

\subsection{Moreau-Yosida regularization of a convex function}

By $\nabla \varphi _{\varepsilon }$ we denote the gradient of the Yosida's
regularization $\varphi _{\varepsilon }$ of the convex lower semicontinuous function $\varphi $, that is
\begin{align*}
\varphi _{\varepsilon }(x)& =\inf \,\{\frac{1}{2\varepsilon }%
|z-x|^{2}+\varphi (z):z\in \mathbb{R}^{d}\} \\
& =\dfrac{1}{2\varepsilon }|x-J_{\varepsilon }x|^{2}+\varphi (J_{\varepsilon
}x),
\end{align*}%
where $J_{\varepsilon }x=x-\varepsilon \nabla \varphi _{\varepsilon }(x).$
The function $\varphi _{\varepsilon }:\mathbb{R}^{d}\rightarrow \mathbb{R}$
is convex and differentiable and, for all $x,y\in \mathbb{R}^{d},$ $%
\varepsilon >0:$%
\begin{equation}
\begin{array}{ll}
(a)\quad & \nabla \varphi _{\varepsilon }(x)=\partial \varphi _{\varepsilon
}\left( x\right) \in \partial \varphi (J_{\varepsilon }x)\textrm{ and }\varphi
(J_{\varepsilon }x)\leq \varphi _{\varepsilon }(x)\leq \varphi
(x),\smallskip \smallskip \\
(b)\quad & \left\vert \nabla \varphi _{\varepsilon }(x)-\nabla \varphi
_{\varepsilon }(y)\right\vert \leq \dfrac{1}{\varepsilon }\left\vert
x-y\right\vert,\smallskip \smallskip \\
(c)\quad & \left\langle \nabla \varphi _{\varepsilon }(x)-\nabla \varphi
_{\varepsilon }(y),x-y\right\rangle \geq 0,\smallskip \smallskip \\
(d)\quad & \left\langle \nabla \varphi _{\varepsilon }(x)-\nabla \varphi
_{\delta }(y),x-y\right\rangle \geq -(\varepsilon +\delta )\left\langle
\nabla \varphi _{\varepsilon }(x),\nabla \varphi _{\delta }(y)\right\rangle .%
\end{array}
\label{sub6a}
\end{equation}%
Moreover, in the case $0=\varphi \left( 0\right) \leq \varphi \left( x\right) $, for
all $x\in \mathbb{R}^{d}$, we have%
\begin{equation}
\begin{array}{l}
\left( a\right) \quad \quad 0=\varphi _{\varepsilon }(0)\leq \varphi
_{\varepsilon }(x)\quad \textrm{and}\quad J_{\varepsilon }\left( 0\right)
=\nabla \varphi _{\varepsilon }\left( 0\right) =0,\smallskip \\
\left( b\right) \quad \quad \dfrac{\varepsilon }{2}|\nabla \varphi
_{\varepsilon }(x)|^{2}\leq \varphi _{\varepsilon }(x)\leq \left\langle
\nabla \varphi _{\varepsilon }(x),x\right\rangle,\quad \forall x\in \mathbb{%
R}^{d}.%
\end{array}
\label{sub6c}
\end{equation}

\begin{proposition}
\label{p12annexB}Let\ $\varphi :\mathbb{R}^{d}\rightarrow ]-\infty ,+\infty
] $ be a proper convex l.s.c. function such that $int\left( Dom\left(
\varphi \right) \right) \neq \emptyset .$ Let $\left( u_{0},\hat{u}%
_{0}\right) \in \partial \varphi,$ $r_{0}\geq 0$ and%
$$
\varphi _{u_{0},r_{0}}^{\#}\overset{def}{=}\sup \left\{ \varphi \left(
u_{0}+r_{0}v\right) :\left\vert v\right\vert \leq 1\right\} .
$$
Then, for all $\,0\leq s\leq t$ and $dk\left( t\right) \in \partial \varphi
\left( x\left( t\right) \right) \left( dt\right) $,%
\begin{equation}
r_{0}\left( \left\updownarrow k\right\updownarrow _{t}-\left\updownarrow
k\right\updownarrow _{s}\right) +\displaystyle\int_{s}^{t}\varphi
(x(r))dr\leq \displaystyle\int_{s}^{t}\left\langle x\left( r\right)
-u_{0},dk\left( r\right) \right\rangle +\left( t-s\right) \varphi
_{u_{0},r_{0}}^{\#}~,  \label{Ba6a}
\end{equation}%
and, moreover,%
\begin{equation}
\begin{array}{l}
r_{0}\left( \left\updownarrow k\right\updownarrow _{t}-\left\updownarrow
k\right\updownarrow _{s}\right) +\displaystyle\int_{s}^{t}\left\vert \varphi
(x(r))-\varphi \left( u_{0}\right) \right\vert dr\leq \displaystyle%
\int_{s}^{t}\left\langle x\left( r\right) -u_{0},dk\left( r\right)
\right\rangle \smallskip \smallskip \\
\quad \quad \quad +\displaystyle\int_{s}^{t}\left( 2\left\vert \hat{u}%
_{0}\right\vert \left\vert x(r)-u_{0}\right\vert +\varphi
_{u_{0},r_{0}}^{\#}-\varphi \left( u_{0}\right) \right) dr.%
\end{array}
\label{Ba6b}
\end{equation}
\end{proposition}

\subsection{Useful inequalities}

Let now introduce the spaces that will appear in the next results.\smallskip

Denote by $S_{d}^{p}\left[0,T\right]$, $p\geq 0$, the space of
progressively measurable continuous stochastic processes $X:\Omega \times %
\left[0,T\right]\rightarrow\mathbb{R}^{d}$, such that%
\[
\left\Vert X\right\Vert _{S_{d}^{p}}=\left\{
\begin{array}{ll}
\left(\mathbb{E}\left\Vert X\right\Vert_{T}^{p}\right)^{\frac{1}{p}\wedge
1}<{\infty}, & \;\text{if }p>0,\bigskip  \\
\mathbb{E}\left[1\wedge\left\Vert X\right\Vert_{T}\right], & \;\text{if }
p=0,
\end{array}
\right.
\]
where $\left\Vert X\right\Vert _{T}=\sup_{t\in \left[ 0,T\right]
}\left\vert X_{t}\right\vert$. The space $(S_{d}^{p}\left[ 0,T\right],\left\Vert \cdot \right\Vert_{S_{d}^{p}}),\ p\!\geq 1,$ is a Banach space
and $S_{d}^{p}\left[ 0,T\right]$, $0\leq p<1$, is a complete metric space
with the metric $\rho(Z_{1},Z_{2})=\left\Vert Z_{1}-Z_{2}\right\Vert
_{S_{d}^{p}}$ (when $p=0$ the metric convergence coincides with the
probability convergence).\smallskip

Denote by $\Lambda_{d\times k}^{p}\left( 0,T\right),\ p\in \lbrack 0,{\infty}[$, the space of progressively measurable stochastic processes $Z:{\Omega }\times ]0,T[\rightarrow \mathbb{R}^{d\times k}$ such that
\[
\left\Vert Z\right\Vert _{\Lambda ^{p}}=\left\{
\begin{array}{ll}
\left[ \mathbb{E}\left( \displaystyle\int_{0}^{T}\Vert Z_{s}\Vert ^{2}ds\right) ^{\frac{p}{2}}\right] ^{\frac{1}{p}\wedge 1}, & \;\text{if }p>0,\bigskip  \\
\mathbb{E}\left[ 1\wedge \left(\displaystyle\int_{0}^{T}\Vert Z_{s}\Vert^{2}ds\right)
^{\frac{1}{2}}\right], & \;\text{if }p=0.
\end{array}
\right.
\]%
The space $(\Lambda_{d\times k}^{p}\left(0,T\right),\left\Vert \cdot
\right\Vert_{\Lambda ^{p}}),\ p\geq 1,$ is a Banach space and $\Lambda
_{d\times k}^{p}\left( 0,T\right) $, $0\leq p<1,$ is a complete metric space
with the metric $\rho (Z_{1},Z_{2})=\left\Vert Z_{1}-Z_{2}\right\Vert
_{\Lambda ^{p}}$.

\begin{proposition}
\label{AnC-GGI} Let $x\in BV_{loc}\left( \left[ 0,\infty \right[;\mathbb{R}%
^{d}\right)$ and $V\in BV_{loc}\left( \left[ 0,\infty \right[ ;\mathbb{R}%
\right)$ be continuous functions. Let $R,$ $N:\left[0,\infty \right[
\rightarrow \left[0,\infty \right[$ be two continuous increasing
functions. If%
$$
\left\langle x\left(t\right),dx\left(t\right)\right\rangle\leq dR\left(
t\right) +\left\vert x\left(t\right)\right\vert dN\left(t\right)
+\left\vert x\left(t\right)\right\vert ^{2}dV\left(t\right)
$$
as signed measures on $\left[0,\infty\right[$, then for all $0\leq t\leq
T,$%
\begin{equation}
\left\Vert e^{-V}x\right\Vert_{\left[t,T\right]}\leq 2\left[\left\vert
e^{-V\left( t\right) }x\left(t\right)\right\vert +\left(
\int_{t}^{T}e^{-2V\left( s\right)}dR\left(s\right)\right)
^{1/2}+\int_{t}^{T}e^{-V\left(s\right)}dN\left(s\right)\right].
\label{ineq1-Anex}
\end{equation}
If $R=0$ then, for all $0\leq t\leq s$,
\begin{equation}
|x(s)|\leq e^{V(s)-V(t)}|x(t)|+\int_{t}^{s}e^{V(s)-V(r)}dN(r).
\label{AnC-dxRNV}
\end{equation}
\end{proposition}

\begin{proof}
Let $u_{\varepsilon
}(r)=|x(r)|^{2}e^{-2V(r)}+\varepsilon $, for $\varepsilon >0$. We have as
signed measures on $[0,\infty )$%
\begin{eqnarray*}
du_{\varepsilon }(r) &=&-2e^{-2V(r)}|x(r)|^{2}dV(r)+2e^{-2V(r)}\left\langle
x(r),dx(r)\right\rangle \medskip  \\
&\leq &2e^{-2V(r)}dR(r)+2e^{-2V(r)}|x(r)|dN(r)\medskip  \\
&\leq &2e^{-2V(r)}dR(r)+2e^{-V(r)}\sqrt{u_{\varepsilon }(r)}dN(r).
\end{eqnarray*}

If $R=0$ then%
\begin{equation*}
d\left( \sqrt{u_{\varepsilon }(r)}\right) =\frac{du_{\varepsilon }(r)}{2%
\sqrt{u_{\varepsilon }(r)}}\leq e^{-V(r)}dN(r),
\end{equation*}
and, consequently, for $0\leq t\leq s$, $\sqrt{u_{\varepsilon }(s)}\leq
\sqrt{u_{\varepsilon }(t)}+\int_{t}^{s}e^{-V(r)}dN(r)$, that yields (\ref{AnC-dxRNV}) by passing to limit as $\varepsilon \rightarrow 0$.

If $R\neq 0$ we have%
\[
\begin{array}{l}
e^{-2V(s)}|x(s)|^{2}\medskip  \\
\quad \leq
e^{-2V(t)}|x(t)|^{2}+2\int_{t}^{s}e^{-2V(r)}dR(r)+2%
\int_{t}^{s}e^{-2V(r)}|x(r)|dN(r)\medskip  \\
\quad \leq e^{-2V(t)}|x(t)|^{2}+2\int_{t}^{s}e^{-2V(r)}dR(r)+2\left\Vert
e^{-V}x\right\Vert _{\left[ t,T\right] }\int_{t}^{s}e^{-V(r)}dN(r)\medskip
\\
\quad \leq |e^{-V(t)}x(t)|^{2}+2\int_{t}^{T}e^{-2V(r)}dR(r)+\dfrac{1}{2}%
\left\Vert e^{-V}x\right\Vert _{\left[ t,T\right] }^{2}+2\left(
\int_{t}^{T}e^{-V(r)}dN(r)\right) ^{2}.%
\end{array}%
\]%
Hence, for all $t\leq \tau \leq T$,%
\begin{eqnarray*}
e^{-2V(\tau )}|x(\tau )|^{2} &\leq &\left\Vert e^{-V}x\right\Vert _{\left[
t,T\right] }^{2}\medskip  \\
&\leq &2e^{-2V(t)}|x(t)|^{2}+4\int_{t}^{T}e^{-2V(s)}dR(s)+4\left(
\int_{t}^{T}e^{-V(s)}dN(s)\right) ^{2}
\end{eqnarray*}%
and the result follows.
\end{proof}

Recall, from Pardoux \& R\u{a}\c{s}canu \cite{Pardoux/Rascanu:09}, an
estimate on the local semimartingale $X\in S_{d}^{0}$ of the form%
\begin{equation}
X_{t}=X_{0}+K_{t}+\int_{0}^{t}G_{s}dB_{s},\;\,t\geq 0,\quad \mathbb{P}-a.s.,
\label{AnexC-fsde0}
\end{equation}%
where\medskip \newline
$\lozenge $\quad $K\in S_{d}^{0},\;$ $K \in BV_{loc}\left(\left[
0,\infty \right[;\mathbb{R}^{d}\right),\;K_{0}=0,\;\mathbb{P}
-a.s.,\smallskip $\newline
$\lozenge $\quad $G\in \Lambda _{d\times k}^{0}.$\smallskip

\noindent For $p\geq 1$ denote $m_{p}\overset{def}{=}1\vee \left( p-1\right)$ and we have the following result.

\begin{proposition}
\label{AnexC-p0-fsi} Let $X\in S_{d}^{0}$ be a local semimartingale of the
form (\ref{AnexC-fsde0}). Assume there exist $p\geq 1$ and $V$ a $\mathcal{P}-$m.b-v.c.s.p.$,$ $V_{0}=0,$ such that as signed measures on $[0,\infty[:$
\begin{equation}
\left\langle X_{t},dK_{t}\right\rangle +\frac{1}{2}m_{p}\left\vert
G_{t}\right\vert ^{2}dt\leq |X_{t}|^{2}dV_{t},\;\;\mathbb{P}-a.s..
\label{AnexC-fsde-ip3}
\end{equation}%
Then, for all $\delta \geq 0$, $0\leq t\leq s,$ we have that%
\begin{equation}
\mathbb{E}^{\mathcal{F}_{t}}\frac{\left\vert e^{-V_{s}}X_{s}\right\vert ^{p}%
}{\left( 1+\delta \left\vert e^{-V_{s}}X_{s}\right\vert ^{2}\right) ^{p/2}}%
\leq \frac{\left\vert e^{-V_{t}}X_{t}\right\vert ^{p}}{\left( 1+\delta
\left\vert e^{-V_{t}}X_{t}\right\vert ^{2}\right) ^{p/2}}~,\;\mathbb{P}-a.s..
\label{AnexC-fsde-p1}
\end{equation}
\end{proposition}

\subsection{Tightness results}

The next five results are given without proofs; you can find them in the monograph \cite{Pardoux/Rascanu:09}.

\begin{proposition}
\label{ch1-p1-tight}Let $\left\{X_{t}^{n}:t\geq 0\right\}$, $n\in\mathbb{N}^{\ast}$, be a family of $\mathbb{R}^{d}-$valued continuous stochastic
processes defined on probability space $\left(\Omega,\mathcal{F},\mathbb{P}\right)$. Suppose that, for every $T\geq 0$, there exist $\alpha=\alpha_{T}>0$ and $b=b_{T}\in C\left(\mathbb{R}_{+}\right)$ with $b(0)=0$ (both independent of $n$), such that
$$
\begin{array}{ll}
\left(i\right)\quad & \lim\limits_{N\rightarrow \infty}\left[\sup\limits_{n\in \mathbb{N}^{\ast}}\mathbb{P}(\{\left\vert
X_{0}^{n}\right\vert\geq N\})\right]=0,\medskip \\
\left(ii\right) \quad & \mathbb{E}\left[1\wedge \sup\limits_{0\leq s\leq
\varepsilon}\left\vert X_{t+s}^{n}-X_{t}^{n}\right\vert^{\alpha}\right]
\leq \varepsilon\cdot b(\varepsilon),\,\forall ~\varepsilon>0,n\geq 1,\;t\in \left[0,T\right].
\end{array}
$$
Then $\left\{ X^{n}:n\in\mathbb{N}^{\ast}\right\}$ is tight in $C(\mathbb{R}_{+};\mathbb{R}^{d})$.
\end{proposition}

\begin{proposition}
\label{ch1-lsc-SI} Consider $\varphi:\mathbb{R}^{d}\rightarrow]-\infty,+\infty]$ a l.s.c. function. Let $\left(X,K,V\right)$, $\left(X^{n},K^{n},V^{n}\right),~n\in\mathbb{N}$, be $C\left(\left[0,T\right];
\mathbb{R}^{d}\right)^{2}\times C\left(\left[0,T\right];\mathbb{R}\right)-$valued random variables, such that
$$
\left(X^{n},K^{n},V^{n}\right)
\xrightarrow
[n\rightarrow\infty]{law}
\left(X,K,V\right)
$$
and, for all $0\leq s<t$, and $n\in\mathbb{N}^{\ast}$
$$
\left\updownarrow K^{n}\right\updownarrow _{t}-\left\updownarrow
K^{n}\right\updownarrow _{s}\leq V_{t}^{n}-V_{s}^{n}\;\;a.s.
$$
and%
$$
{\displaystyle\int_{s}^{t}}\varphi \left( X_{r}^{n}\right)dr\leq \displaystyle\int_{s}^{t}\left\langle X_{r}^{n}~,dK_{r}^{n}\right\rangle,\;\;a.s..
$$
Then $\left\updownarrow K\right\updownarrow_{t}-\left\updownarrow
K\right\updownarrow_{s}\leq V_{t}-V_{s}~,\;a.s.$ and
$$
\displaystyle\int_{s}^{t}\varphi\left(X_{r}\right)dr\leq\displaystyle\int_{s}^{t}\left\langle X_{r}~,dK_{r}\right\rangle,\;\;a.s..
$$
\end{proposition}

\begin{proposition}
\label{ch3-c4-cont} Let $X,\hat{X}\in S_{d}^{0}\left[0,T\right]$ and $B,\hat{B}$ be two $\mathbb{R}^{k}-$Brownian motions and $g:\mathbb{R}_{+}\times\mathbb{R}^{d}\rightarrow\mathbb{R}^{d\times k}$ be a function
satisfying
$$
\begin{array}{l}
g\left(\cdot,y\right) \textrm{ is measurable\ }\forall ~y\in \mathbb{R}^{d},\;\;\textrm{and}\medskip \\
y\mapsto g\left(t,y\right) \textrm{is continuous\ }dt-a.e..
\end{array}
$$
If
$$
\mathcal{L}\left( X,B\right) =\mathcal{L}(\hat{X},\hat{B}),\;\;\textrm{on\ }C(\mathbb{R}_{+},\mathbb{R}^{d+k}),
$$
then
$$
\mathcal{L}\left(X,B,\int_{0}^{\cdot}g\left( s,X_{s}\right) dB_{s}\right)=
\mathcal{L}\left(\hat{X},\hat{B},\int_{0}^{\cdot}g\left(s,\hat{X}_{s}\right) d\hat{B}_{s}\right),\;\;\textrm{on\ } C(\mathbb{R}_{+},\mathbb{R}^{d+k+d}).
$$
\end{proposition}

\begin{lemma}
\label{ch1-p2-tight} Let $g:\mathbb{R}_{+}\rightarrow\mathbb{R}_{+}$ be a
continuous function satisfying $g\left(0\right)=0$ and $G:C\left(\mathbb{R}_{+};\mathbb{R}^{d}\right)\rightarrow \mathbb{R}_{+}$ be a mapping which
is bounded on compact subsets of $C\left(\mathbb{R}_{+};\mathbb{R}
^{d}\right).$ Let $X^{n},Y^{n}$, $n\in\mathbb{N}^{\ast}$, be random
variables with values in $C\left(\mathbb{R}_{+};\mathbb{R}^{d}\right)$. If
$\left\{ Y^{n}:n\in \mathbb{N}^{\ast }\right\}$ is tight and, for all $n\in
\mathbb{N}^{\ast},$%
$$
\begin{array}{rl}
\left(i\right) \quad & \left\vert X_{0}^{n}\right\vert \leq G\left(Y^{n}\right),\;a.s.,\medskip \\
\left(ii\right)\quad & \mathbf{m}_{X^{n}}\left(\varepsilon;\left[0,T\right]\right)\leq G\left(Y^{n}\right) g\left(\mathbf{m}_{Y^{n}}\left(\varepsilon;
\left[0,T\right]\right)\right),\;a.s.,\;\;\forall ~\varepsilon,T>0,
\end{array}
$$
then $\left\{ X^{n}:n\in \mathbb{N}^{\ast }\right\}$ is tight.
\end{lemma}

\begin{lemma}
\label{ch2-p-conv}Let $B,$ $B^{n}$, $\bar{B}^{n}:\Omega\times\left[
0,\infty \right[ \rightarrow \mathbb{R}^{k}$ and $X$, $X^{n}$, $\bar{X}^{n}:\Omega \times\left[0,\infty\right[\rightarrow\mathbb{R}^{d\times k}$ be continuous stochastic processes such that

\begin{itemize}
\item[$\left(i\right)$] $\quad B^{n}$ is $\mathcal{F}_{t}^{B^{n},X^{n}}-$ Brownian motion, for all $n\geq 1,$

\item[$\left(ii\right)$] $\quad\mathcal{L}(X^{n},B^{n})=\mathcal{L}\left(
\bar{X}^{n},\bar{B}^{n}\right)$ on $C(\mathbb{R}_{+},\mathbb{R}^{d\times
k}\times \mathbb{R}^{k})$, for all $n\geq 1$,

\item[$\left( iii\right)$] $\quad\displaystyle\int_{0}^{T}\left\vert \bar{X}_{s}^{n}-\bar{X}_{s}\right\vert^{2}ds+\sup\limits_{t\in \left[0,T\right]}\left\vert \bar{B}_{t}^{n}-\bar{B}_{t}\right\vert\longrightarrow 0$ in
probability, as $n\rightarrow \infty$, for all $T>0.$
\end{itemize}

\noindent Then $(\bar{B}^{n},\{\mathcal{F}_{t}^{\bar{B}^{n},\bar{X}^{n}}\}),n\geq 1,\;$and $(\bar{B},\{\mathcal{F}_{t}^{\bar{B},\bar{X}}\})$
are Brownian motions and, as $n\rightarrow\infty$,
$$
\sup_{t\in\left[0,T\right]}\left\vert\int_{0}^{t}\bar{X}_{s}^{n}d\bar{B}_{s}^{n}-\int_{0}^{t}\bar{X}_{s}d\bar{B}_{s}\right\vert
\underset{n\rightarrow \infty}{\longrightarrow}0\quad \textrm{in probability}.
$$
\end{lemma} \medskip

\section*{Acknowledgements}

The authors are grateful to the referees for the attention in reading this paper and for theirs very useful suggestions.

\bibliographystyle{model1-num-names}

\end{document}